\documentclass{article}
\usepackage{amsmath, amsfonts, amsthm, amssymb}
\usepackage{graphicx}
\usepackage{float}
\usepackage{verbatim}

\numberwithin{equation}{section}
\newtheorem{thm}{Theorem}[section]
\newtheorem{lma}[thm]{Lemma}
\newtheorem{cor}[thm]{Corollary}
\newtheorem{defn}[thm]{Definition}

\newtheorem{prop}[thm]{Proposition}

\newtheorem{ques}[thm]{Question}

\renewcommand{\ge}{\geqslant}

\renewcommand{\geq}{\geqslant}
\renewcommand{\leq}{\leqslant}
\renewcommand{\H}{\text{H}}
\renewcommand{\P}{\text{P}}

\allowdisplaybreaks

\title{On the $L^q$-spectrum of planar self-affine measures}

\author{Jonathan M. Fraser\\ \\
\emph{Mathematics Institute, Zeeman Building,}\\ \emph{University of Warwick, Coventry, CV4 7AL, UK.}\\ \emph{e-mail: jon.fraser32@gmail.com}}

\begin{document}
\maketitle

\begin{abstract}
We study the dimension theory of a class of planar self-affine multifractal measures.  These measures are the Bernoulli measures supported on box-like self-affine sets, introduced by the author, which are the attractors of iterated function systems consisting of contracting affine maps which take the unit square to rectangles with sides parallel to the axes.  This class contains the self-affine measures recently considered by Feng and Wang as well as many other measures.  In particular, we allow the defining maps to have non-trivial rotational and reflectional components. Assuming the rectangular open set condition, we compute the $L^q$-spectrum by means of a $q$-modified singular value function

A key application of our results is a \emph{closed form expression} for the $L^q$-spectrum in the case where there are no mappings that switch the coordinate axes.  This is useful for computational purposes and also allows us to prove differentiability of the $L^q$-spectrum at $q=1$ in the more difficult `non-multiplicative' situation.  This has applications concerning the Hausdorff, packing and entropy dimension of the measure as well as the Hausdorff and packing dimension of the support.  Due to the possible inclusion of axis reversing maps, we are led to extend some results of Peres and Solomyak on the existence of $L^q$-spectrum of self-similar measures to the graph-directed case.
\end{abstract}

\setcounter{tocdepth}{2}
\tableofcontents


\section{Introduction}

Since smooth non-conformal dynamics are locally self-affine, self-affine sets and measures are an important building block towards understanding naturally occurring physical phenomena exhibiting a non-conformal structure and have thus attracted a great deal of attention in both pure and applied fields. In this paper we conduct a detailed dimensional analysis of a new class of planar self-affine multifractal measures, which are the Bernoulli measures supported on box-like self-affine sets, see \cite{me_box}.  Our measures generalise those considered by King \cite{king} and, in particular, by Feng and Wang \cite{fengaffine} by allowing the mappings in the IFS to have non-trivial rotational and reflectional components.  We introduce a $q$-modified singular value function, which is a natural multifractal generalisation of the modified singular value function defined in \cite{me_box} and, assuming a natural rectangular open set condition, we compute the $L^q$-spectrum. The $L^q$-spectrum of a measure $\mu$ is an important quantitative description of the global fluctuations of $\mu$ and has been linked to many geometric characteristics of the measure, mostly concerning dimension. 
\\ \\
A major difference between our class of sets and measures and their orientation preserving counterparts is brought to light in one of our main applications.  In the `axis preserving case' we are able to derive a closed form expression for the $L^q$-spectrum, assuming a mild technical assumption.  This is particularly important for computational purposes as one does not have to rely on `$k$th level approximations' which get increasingly more computationally inefficient to compute.  It also allows us to obtain precise results on the differentiability of the $L^q$-spectrum, which has applications concerning Hausdorff dimension, see \cite{ngai}.  It was observed in \cite{fengaffine} that if the maps are orientation preserving and the stronger contraction is always in the same direction, then the $L^q$-spectrum is differentiable at $q=1$ provided the $L^q$-spectra of the relevant projections are differentiable at this value. Using some arguments from convex analysis, we extend this to include the `non-multiplicative' situation where the maps do not have to be orientation preserving and the stronger contraction does not always have to be in the same direction.  This allows us to compute the Hausdorff dimension of a much larger class of self-affine measures. As a second application, we note that the Legendre transform of the $L^q$-spectrum gives an upper bound for the multifractal Hausdorff and packing spectra, and can be used to compute the packing dimension, and give non-trivial lower bounds for the Hausdorff dimension of the support of the measure.
\\ \\
A second evident difference is that the orthogonal projections of the measures need not be self-similar.  Indeed, if one of the defining maps sends vertical lines to horizontal lines and vice versa, then the orthogonal projections of the measure are a pair of intertwined \emph{graph-directed} self-similar measures.  This introduces some technical difficulties in the proofs and, in order to obtain our results, we are forced to prove that the $L^q$-spectrum of such measures exists, regardless of separation conditions, thus generalizing results of Peres and Solomyak concerning the 1-vertex case \cite{exists}.  This may have useful applications in other situations.

\subsection{The $L^q$-spectrum and the multifractal formalism}

Let $\mathcal{P}(\mathbb{R}^n)$ be the set of compactly supported regular Borel probability measures on $\mathbb{R}^n$ and let $\mu \in \mathcal{P}(\mathbb{R}^n)$ with support denoted by $\text{supp} \, \mu$. The \emph{lower and upper $L^q$-spectrum} of $\mu$ are defined by
\[
\underline{\tau}_\mu(q) = \underline{\lim}_{\delta \to 0} \frac{\log \int_{\text{supp} \, \mu} \mu\big(B(x,\delta)\big)^{q-1} \, d\mu(x)}{-\log \delta}
\]
and
\[
\overline{\tau}_\mu(q) = \overline{\lim}_{\delta \to 0} \frac{\log \int_{\text{supp} \, \mu} \mu\big(B(x,\delta)\big)^{q-1}\, d\mu(x)}{-\log \delta}
\]
respectively, with $q \in \mathbb{R}$. If $\underline{\tau}_\mu(q) = \overline{\tau}_\mu(q)$, then we write $\tau_\mu(q)$ to denote the common value. The $L^q$-spectrum is sometimes referred to as the \emph{$L^q$-dimensions} or the \emph{generalised dimensions}.   It is a simple consequence of H\"older's inequality that $\underline{\tau}_\mu(q)$ and $\overline{\tau}_\mu(q)$ are convex and thus continuous on $(0,\infty)$.  It is also easy to see that they are decreasing on $[0,\infty)$ and equal to 0 at $q=1$.  Moreover, they are Lipschitz continuous on $[t,\infty)$ for any $t>0$ and differentiable on $(0,\infty)$ at all but at most countably many points.  It is related to the packing and covering multifractal box dimensions, which are alternative attempts to describe the global fluctuations of a measure, see \cite{multifractalformalism} for the definitions.  Write $\underline{T}_{\P, \mu}(q)$ and $\overline{T}_{\P, \mu}(q)$ for the lower and upper packing multifractal box dimensions of $\mu$, and $\underline{T}_{\text{C}, \mu}(q)$ and $\overline{T}_{\text{C}, \mu}(q)$ for the lower and upper covering multifractal box dimensions of $\mu$.
\begin{prop}[Relationships between the dimension functions]
Let $\mu \in \mathcal{P}(\mathbb{R}^n)$.  Then
\begin{itemize}
\item[(1)] For $q<0$, we have
\[
\underline{\tau}_\mu(q) \leq \underline{T}_{\text{C}, \mu}(q) = \underline{T}_{\P, \mu}(q)
\]
and
\[
\overline{\tau}_\mu(q) \leq \overline{T}_{\text{C}, \mu}(q) = \overline{T}_{\P, \mu}(q);
\]
\item[(2)] For $q \in [0,1]$, we have
\[
\underline{\tau}_\mu(q) = \underline{T}_{\text{C}, \mu}(q) = \underline{T}_{\P, \mu}(q)
\]
and
\[
\overline{\tau}_\mu(q) = \overline{T}_{\text{C}, \mu}(q) = \overline{T}_{\P, \mu}(q);
\]
\item[(3)] For $q>1$, we have
\[
\underline{T}_{\text{C}, \mu}(q) \leq  \underline{T}_{\P, \mu}(q) = \underline{\tau}_\mu(q)
\]
and
\[
\overline{T}_{\text{C}, \mu}(q)  \leq  \overline{T}_{\P, \mu}(q) = \overline{\tau}_\mu(q);
\]
\item[(4)] If $\mu$ is doubling, then, for all $q \in \mathbb{R}$, we have
\[
\underline{T}_{\text{C}, \mu}(q) =  \underline{T}_{\P, \mu}(q) = \underline{\tau}_\mu(q)
\]
and
\[
\overline{T}_{\text{C}, \mu}(q)  =  \overline{T}_{\P, \mu}(q) = \overline{\tau}_\mu(q).
\]
\end{itemize}
\end{prop}

\begin{proof}
For $q\geq 0$, see \cite{exists} and the references therein.  For $q<0$, see, for example, \cite{multifractalformalism, multigeom}.  Part (4) is discussed in \cite{multigeom}.
\end{proof}

One of the key properties of the dimension functions discussed above is their relationship with multifractal spectra.  With the desire to study the local structure of $\mu$ we form \emph{multifractal decomposition sets} $\Delta_\alpha$, defined by
\begin{equation}
\Delta_\alpha = \Big\{x \in F  :  \dim_{\text{loc}} \mu(x)  = \alpha  \Big\}
\end{equation}
for $\alpha \ge 0$, where $\dim_{\text{loc}} \mu(x)$ is the local dimension of $\mu$ at $x$, if it exists.  Letting $\Delta_\text{x}$ denote the set of points which do not have a local dimension we have formed a multifractal decomposition of the support of $\mu$, $\text{supp} \mu = \Big(\bigcup_{\alpha\geq 0}\Delta_\alpha\Big) \cup \Delta_\text{x}$.  The Hausdorff and packing multifractal spectrum functions, $f_{\H,\mu}$ and $f_{\P,\mu}$, are defined by
\[
 f_{\H,\mu}(\alpha) =   \dim_{\H}  \Delta_\alpha
\]
and
\[
 f_{\P,\mu}(\alpha) =   \dim_{\P}  \Delta_\alpha
\]
for $\alpha \ge 0$, where $\dim_{\H}$ and $\dim_{\P}$ denote the Hausdorff and packing dimensions respectively.  The Hausdorff and packing spectrum are often very difficult to compute and are only known for a few classes of measures, mostly with a self-conformal structure.  A common approach is to study the validity of the \emph{multifractal formalism}, which states, roughly speaking, that the multifractal spectra are equal to the Legendre transform $(\cdot^*)$ of an appropriately defined moment scaling function.

\begin{prop} \label{upperbounds}
Let $\mu \in \mathcal{P}(\mathbb{R}^n)$ be compactly supported.  Then, for all $\alpha \geq 0$, we have
\[
\begin{array}{ccccc}
& &                                                                                f_{\text{\emph{P}},\mu}(\alpha)                                    & &   \\
 &                                 \rotatebox[origin=c]{45}{$\leq$}          & &                \rotatebox[origin=c]{315}{$\leq$} &   \\
 f_{\text{\emph{H}},\mu}(\alpha)                                          &  &            & &             \overline{T}^*_{\P,\mu}(\alpha)                          \\
 &                                 \rotatebox[origin=c]{315}{$\leq$}       & &                  \rotatebox[origin=c]{45}{$\leq$} &  \\
& &                                                                        \underline{T}^*_{\P,\mu}(\alpha)                        & & 
\end{array}
\]
\end{prop}

\begin{proof}
See, for example, \cite[Theorem 3.3.1]{multigeom}.
\end{proof}

If it exists, the function $T^*_{\P,\mu}(\alpha) $ will be referred to as the \emph{Legendre spectrum} of $\mu$.  Another important application of the $L^q$-spectrum concerns the dimension theory of the measure and the support of the measure.  We write $\dim_\H$, $\dim_\P$ and $\dim_{\text{e}}$ to denote the Hausdorff, packing and entropy dimension respectively.
\begin{prop} \label{dimapps}
Let $\mu \in \mathcal{P}(\mathbb{R}^n)$.  Then
\[
\overline{\dim}_{\text{\emph{B}}} \, \text{\emph{supp}}(\mu) = \overline{\tau}(0)
\]
and
\[
\underline{\dim}_{\text{\emph{B}}} \, \text{\emph{supp}}(\mu) = \underline{\tau}(0).
\]
Furthermore, if $\overline{\tau}$ is differentiable at $q=1$, then
\[
\dim_{\text{\emph{H}}}  \mu = \dim_{\text{\emph{P}}}  \mu = \dim_{\text{\emph{e}}} \mu =- \overline{\tau}'(1)
\]
and therefore
\[
- \overline{\tau}'(1) \leq \dim_{\text{\emph{H}}}  \text{\emph{supp}}(\mu) \leq \dim_{\text{\emph{P}}}  \text{\emph{supp}}(\mu) \leq   \overline{\tau}(0).
\]
\end{prop}

\begin{proof}
The box dimension result is obvious and the dimension results for $\mu$ are due to Ngai \cite{ngai}, see also \cite{young}.  Finally, the Hausdorff and packing dimension estimates for $\text{supp}(\mu)$ follow from basic properties of dimensions, see \cite{falconer}.
\end{proof}

If one is interested in the Hausdorff dimension of a dynamically defined set $F$ modelled by a full shift, then it is natural to compute the Hausdorff dimension of Bernoulli measures supported on $F$.  This always gives a lower bound and often taking the supremum over all such measures attains the correct value.

\subsection{Self-affine measures}

Given an iterated function system (IFS) $\{S_i\}_{i \in \mathcal{I}}$, for some finite index set $\mathcal{I}$, where each of the maps $S_i$ is an affine contracting self map on $\mathbb{R}^n$, i.e., a linear contraction composed with a translation, and an associated probability vector $\{p_i\}_{i \in \mathcal{I}}$ with each $p_i \in (0,1)$, there is a unique non-empty compact set $F$ satisfying
\[
F = \bigcup_{i \in \mathcal{I}} S_i(F),
\]
which is called the \emph{self-affine} attractor of the IFS, a unique Borel probability measure $\mu$ satisfying
\[
\mu = \sum_{i \in \mathcal{I}} p_i \, \mu \circ S_{i}^{-1},
\]
which is called the \emph{self-affine measure} associated with $\{S_i\}_{i \in \mathcal{I}}$ and $\{p_i\}_{i \in \mathcal{I}}$.  It is easy to see that the support of $\mu$ is equal to $F$.  King \cite{king} computed the Hausdorff multifractal spectrum for the self-affine measures supported on Bedford-McMullen carpets assuming a strong separation condition, which was later removed by Jordan and Rams \cite{jordanrams}.  Olsen \cite{sponges} studied various multifractal properties of self-affine multifractal sponges (the higher dimensional analogue of the Bedford-McMullen carpet).  In particular, he computed the Hausdorff spectrum and the multifractal box dimensions assuming the higher dimensional analogue of King's separation condition - known as the \emph{very strong separation condition}.  Feng and Wang \cite{fengaffine} computed the $L^q$-spectrum in the range $q \geq 0$ for a considerably more general class of planar self-affine measures where the maps in the IFS take the unit square onto a rectangle with sides parallel to the axes, with the orientation preserved.  Ni and Wen \cite{gdaffine} extended these results to the graph-directed situation in some restricted cases.  Random self-affine measures were studied in \cite{randomsponges, me_random}.  In the context of more general self-affine sets, Falconer computed the $L^q$-spectrum in the range [0,1] for generic self-affine measures \cite{genericaffinemeasure} and in the range $[0,\infty)$ for almost self-affine measures \cite{almostselfaffine}.  Barral and Feng \cite{barralfeng} considered the very difficult question of whether the multifractal formalism holds generically for self-affine measures and made some significant progress concerning the part of the spectrum corresponding to $q \in (1,2)$.  To the best of our knowledge the results of Olsen \cite{sponges} are the only example where the $L^q$-spectrum has been computed for non-trivial self-affine measures in the range $q<0$.  This is a notoriously difficult problem, partly due to standard inequalities like Jensen and H\"older going in the wrong direction for $q<0$.  We discuss this problem in Section \ref{negativeq}.

\subsection{Our class of measures and some notation} \label{boxdef}

Box-like self-affine sets were defined in \cite{me_box}.  They are the attractors of IFSs consisting of contracting affine maps which take the unit square to a subrectangle with sides parallel to the axes.  The affine maps which make up such an IFS are necessarily of the form $S_i = T_i \circ L_i + t_i$, where $T_i$ is a contracting linear map of the form
\[
T_i = \left ( \begin{array}{cc}
c_i & 0\\ 
0 &  d_i\\
\end{array} \right ) 
\]
for some $c_i, d_i \in (0,1)$; $L_i$ is a linear isometry of the plane for which $L_i([-1,1]^2) =[-1,1]^2$; and $t_i \in \mathbb{R}^2$ is a translation vector.  If for all $i \in \mathcal{I}$, $L_i$ is the identity map, then we obtain the class of self-affine sets considered by Feng and Wang \cite{fengaffine}.  Let $\{S_i \}_{i \in \mathcal{I}}$ be an IFS consisting of maps of the form described above for some finite index set $\mathcal{I}$, with $\lvert \mathcal{I} \rvert \geq 2$, and let $\{p_i \}_{i \in \mathcal{I}}$ be a corresponding probability vector with $p_i \in (0,1)$ for all $i \in \mathcal{I}$.  Let $F$ be the associated self-affine set and $\mu$ be the associated self-affine measure with $\text{supp} \, \mu = F$.  The following separation condition, which we will need to obtain some of our results, was introduced in \cite{fengaffine} and also used in \cite{me_box}.

\begin{defn}
An IFS $\{S_i\}_{i \in \mathcal{I}}$ satisfies the rectangular open set condition (ROSC) if there exists a non-empty open rectangle, $R = (a,b)\times(c,d) \subset \mathbb{R}^2$, such that $\{S_i(R)\}_{i \in \mathcal{I}}$ are pairwise disjoint subsets of $R$.
\end{defn}
Let
\[
\mathcal{I}_A = \{ i \in \mathcal{I} : S_i \text{ maps horizontal lines to horizontal lines} \}
\]
and
\[
\mathcal{I}_B = \{ i \in \mathcal{I} : S_i \text{ maps horizontal lines to vertical lines} \}.
\]
If $\mathcal{I}_B=\emptyset$, then we will say $\mu$ is of \emph{separated type} and otherwise we will say that $\mu$ is of \emph{non-separated type}.  It will become clear why we make this distinction in the following section. Write $\mathcal{I}^* = \bigcup_{k\geq1} \mathcal{I}^k$ to denote the set of all finite sequences with entries in $\mathcal{I}$ and for $\textbf{\emph{i}}= \big(i_1, i_2, \dots, i_k \big) \in \mathcal{I}^*$ write $S_{\textbf{\emph{i}}} = S_{i_1} \circ S_{i_2} \circ \dots \circ S_{i_k}$ and $p(\textbf{\emph{i}}) = p_{i_1} \cdots p_{i_k}$.  The \emph{singular values} of a linear map $A$ are the positive square roots of the eigenvalues of $A^T A$.  Viewed geometrically, these numbers are the lengths of the semi-axes of the image of the unit ball under $A$.  Thus, roughly speaking, the singular values correspond to how much the affine map contracts (or expands) in different directions.  Write $\alpha_1 (\textbf{\emph{i}}) \geq \alpha_2 (\textbf{\emph{i}})$ for the singular values of the linear part of the map $S_{\textbf{\emph{i}}}$.  Note that, for all $\textbf{\emph{i}} \in \mathcal{I}^*$, the singular values, $\alpha_1 (\textbf{\emph{i}})$ and $\alpha_2 (\textbf{\emph{i}})$, are just the lengths of the sides of the rectangle $S_{\textbf{\emph{i}}}\big([0,1]^2\big)$.  Let
\[
\alpha_{\min} = \min \{\alpha_2(i) : i \in \mathcal{I} \},
\]
\[
\alpha_{\max} = \max \{\alpha_1(i) : i \in \mathcal{I} \},
\]
\[
p_{\min} = \min \{p_i : i \in \mathcal{I} \}
\]
and
\[
p_{\max} = \max \{p_i : i \in \mathcal{I} \}
\]
and note that $0<\alpha_{\min}, \alpha_{\max}, p_{\min}, p_{\max} <1$.  In \cite{me_box} the box and packing dimensions of box-like self-affine sets were computed (assuming the ROSC) by means of a modified singular value function.  In the following section we introduce a natural multifractal analogue which we call a $q$-modified singular value function.
\\ \\
Box-like self affine sets (and measures) can enjoy a much richer visual structure than other classes of self-affine carpet.  This is because the non-trivial rotations can compound to create very natural looking images, more akin to general self-affine constructions.

\begin{figure}[H] 
	\centering
	\includegraphics[width=160mm]{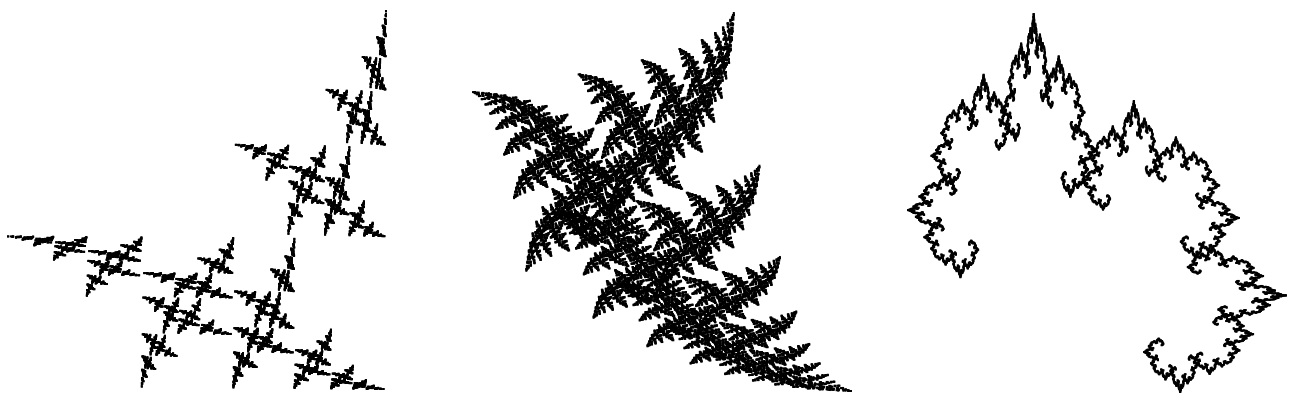}
\caption{Some examples of box-like self-affine supports.}
\end{figure}

\section{Results} \label{results}

\subsection{Existence of the $L^q$-spectrum for graph-directed self-similar measures} \label{existence}

In this section we will discuss the $L^q$-spectrum of self-similar and graph-directed self-similar measures.  This is relevant for our work because the $L^q$-spectrum of our class of self-affine measures depends on the $L^q$-spectra of the projections of the measure onto the horizontal and vertical axes.  These measures are either self-similar measures or a pair of graph-directed self-similar measures.  The defining IFSs for these measures may have complicated overlaps and so computing the exact $L^q$-spectrum is currently intractable, but the important thing for us is that they exist, so that we can compute the $L^q$-spectrum of the self-affine measure \emph{in terms} of the projected measures.  It was shown by Peres and Solomyak \cite{exists} that the $L^q$-spectrum exists for positive $q$ for any self-similar (even self-conformal) measure regardless of separation conditions.  However, we must extend these results to the graph-directed case.  The paper \cite{gdaffine} considered graph-directed self-affine measures and the $L^q$-spectrum depended on the $L^q$-spectra of the projected measures which were graph-directed self-similar measures.  However, there the authors assume that these $L^q$-spectra exist, without proving it explicitly.
\\ \\
Let $\Gamma = (\mathcal{V}, \mathcal{E})$ be a finite connected directed multigraph with vertices $\mathcal{V} = \{1, \dots, N\}$ and a finite multiset of edges $ \mathcal{E}$.  Write $\mathcal{E}_{i,j}$ for the multiset of all edges joining the vertex $i$ to the vertex $j$ and for each $e \in \mathcal{E}$ associate a contracting similarity mapping $S_e$ from $\mathbb{R}^d \to \mathbb{R}^d$ with similarity ratio $c_e \in (0,1)$, and a probability $p_e \in (0,1)$ such that
\[
\sum_{j = 1}^{N} \sum_{e \in \mathcal{E}_{i,j}} p_e = 1
\]
for all $ i \in \mathcal{V}$.  It is standard that their exists a unique family of non-empty compact sets $\{F_{i}\}_{i  \in \mathcal{V}}$ satisfying
\[
F_i = \bigcup_{j = 1}^{N} \bigcup_{e \in \mathcal{E}_{i,j}} S_e(F_j)
\]
and a unique family of Borel probability measures $\{\mu_{i}\}_{i  \in \mathcal{V}}$ satisfying
\[
\mu_i = \sum_{j = 1}^{N} \sum_{e \in \mathcal{E}_{i,j}} p_e \mu_j \circ S_e^{-1}.
\]
Furthermore, the support of $\mu_i$ is $F_i$ for all $i \in \mathcal{V}$.  The measures $\{\mu_{i}\}_{i  \in \mathcal{V}}$ are called \emph{a family of graph-directed self-similar measures}.  The main result of this section is the following.
\begin{thm} \label{gdexists}
For every family of graph-directed self-similar measures $\{\mu_{i}\}_{i  \in \mathcal{V}}$, the $L^q$-spectrum exists for all $q\geq0$ and are the same for each measure $\mu_{i}$,  i.e. for all $q \geq 0$ and $i,j \in \mathcal{V}$, we have
\[
\underline{\tau}_{\mu_i}(q) = \overline{\tau}_{\mu_i}(q) = \underline{\tau}_{\mu_j}(q) = \overline{\tau}_{\mu_j}(q).
\]
\end{thm}

We will prove Theorem \ref{gdexists} in Section \ref{gdproof}.

\subsection{The $L^q$-spectrum of projected measures} \label{projections}

Our results rely on knowledge of the $L^q$-spectra of the projections of $\mu$ onto the orthogonal axes.  Let $\pi_1, \pi_2: \mathbb{R}^2 \to \mathbb{R}$ be defined by $\pi_1(x,y) = x$ and $\pi_2(x,y) = y$ respectively.  It follows from \cite[Lemma 2.8]{me_box} that if $\mu$ is of separated type, then $\pi_1(\mu)$ and $\pi_2(\mu)$ are self-similar measures, and otherwise, they are a pair of graph-directed self-similar measures, hence the relevance of the previous section.  For $q \geq 0$, let
\[
\tau_1(q) = \tau_{\pi_1(\mu)}(q)
\]
and
\[
\tau_2(q) =\tau_{\pi_2(\mu)}(q).
\]
It follows from Theorem \ref{gdexists} that $\tau_1(q)$ and $\tau_2(q)$ exist for all $q \geq 0$ and, moreover, if $\mu$ is non-separated, then $\tau_1 \equiv \tau_2$. The problem with calculating the dimension of $\tau_1(q) $ and $\tau_2(q) $ is that the IFSs of similarities alluded to above may not satisfy the open set condition (OSC), or graph-directed open set condition (GDOSC).  If $\mu$ is of separated type, then the natural candidate for the $L^q$-spectra of $\pi_1(\mu)$ and $\pi_2(\mu)$ are given by a simple standard formula, see \cite[Chapter 17, (17.26)]{falconer}.  If $\mu$ is of non-separated type then the situation is slightly more complicated.  First one defines an associated weighted adjacency matrix $A^{(q,t)}$ and then the natural candidate for the $L^q$-spectrum of the measures $\pi_1(\mu)$ and $\pi_2(\mu)$ is the function $\beta:\mathbb{R} \to \mathbb{R}$ defined by
\[
\rho\big(A^{(q,t)} \big) = 1,
\]
where $\rho\big(A^{(q,t)} \big)$ denotes the spectral radius of $A^{(q,t)}$, see \cite{stric}.  The basic concept is that the `natural candidates' actually give the $L^q$-spectrum provided the underlying IFSs have enough separation.  This problem has been considered by many authors, in particular, Strichartz \cite{stric}, Falconer \cite{falconer}, Olsen \cite{multifractalformalism, olsenbook} and Riedi \cite{riedi}.  Rather than state results explicitly, we adopt the philosophy that in certain `nice' cases with enough `separation' the natural candidates give us the correct function, but we are much more focused on the fact that the $L^q$-spectrum exists and therefore can be used to state our main results.
\\ \\
We will occasionally require that $\tau_1$ and $\tau_2$ are differentiable.  We conclude this section by observing that the `natural candidates' for $\tau_1$ and $\tau_2$ discussed above are differentiable for all $q >0$ and, moreover, Feng \cite{fengsmooth} recently proved that the $L^q$-spectrum of a self-similar measure on the line is differentiable for $q>0$ in certain overlapping cases.  However, Barral and Feng have recently shown that for any $q_0 \in (1,2)$ it is possible to construct a self-similar measure with overlaps for which the $L^q$-spectrum is not differentiable at $q_0$, see \cite[Remark 6.8]{barralfeng}.  In Section \ref{example2}, we provide an example of a \emph{self-affine} measure for which the $L^q$-spectrum is not differentiable at a point in $(0,1)$.

\subsection{A moment scaling function $\gamma(q)$} \label{gammadefsection}

For $\textbf{\emph{i}} \in \mathcal{I}^*$, let $b(\textbf{\emph{i}}) = \lvert \pi_1(S_{\textbf{\emph{i}}}[0,1]^2)\rvert$ and $h(\textbf{\emph{i}}) = \lvert \pi_2(S_{\textbf{\emph{i}}}[0,1]^2)\rvert$ denote the length of the base and height of the rectangle $S_{\textbf{\emph{i}}}[0,1]^2$ respectively and define $\pi_{\textbf{\emph{i}}}:\mathbb{R}^2 \to \mathbb{R}$ by
\[
\pi_{\textbf{\emph{i}}} = \left\{ \begin{array}{cc}
\pi_1 &  \text{if $\textbf{\emph{i}} \in \mathcal{I}_A$ and $b(\textbf{\emph{i}}) \geq h(\textbf{\emph{i}})$}\\ 
\pi_2 &   \text{if $\textbf{\emph{i}} \in \mathcal{I}_A$ and $b(\textbf{\emph{i}}) < h(\textbf{\emph{i}})$}\\
\pi_1 &   \text{if $\textbf{\emph{i}} \in \mathcal{I}_B$ and $b(\textbf{\emph{i}}) < h(\textbf{\emph{i}})$}\\ 
\pi_2 &   \text{if $\textbf{\emph{i}} \in \mathcal{I}_B$ and $b(\textbf{\emph{i}}) \geq h(\textbf{\emph{i}})$}
\end{array} \right. 
\]
Finally, let $\tau_\textbf{\emph{i}}(q) = \tau_{\pi_{\textbf{\emph{i}}}\mu}(q)$.  In fact, $\tau_\textbf{\emph{i}}(q)$ is simply the $L^q$-spectrum of the projection of $\mu \vert_{S_\textbf{\emph{i}}(F)}$ onto the longest side of the rectangle $S_\textbf{\emph{i}}\big([0,1]^2\big)$ and is always equal to either $\tau_1(q)$ or $\tau_2(q)$.  For $s \in \mathbb{R}$ and $q \geq 0$, define the \emph{$q$-modified singular value function} $\psi^{s,q}:\mathcal{I}^* \to (0,\infty)$ by
\begin{equation} \label{modsing}
\psi^{s,q}\big({\textbf{\emph{i}}}\big) = p(\textbf{\emph{i}})^q \, \, \alpha_1 (\textbf{\emph{i}})^{ \tau_\textbf{\emph{i}}(q)} \, \,  \alpha_2 (\textbf{\emph{i}})^{s-\tau_\textbf{\emph{i}}(q)},
\end{equation}

and for $k \in \mathbb{N}$, define a number $\Psi_k^{s,q}$ by
\[
\Psi_k^{s,q}= \sum_{\textbf{\emph{i}} \in \mathcal{I}^{k}} \psi^{s,q}({\textbf{\emph{i}}}).
\]
Note that $\psi^{s,q}$ and $\Psi_k^{s,q}$ are multifractal analogues of $\psi^{s}$ and $\Psi_k^{s}$, defined in \cite{me_box}.  In fact, it is easy to see that $\psi^{s}(S_\textbf{\emph{i}}) = \psi^{s,0}(\textbf{\emph{i}})$ and $\Psi_k^{s} = \Psi_k^{s,0}$.

\begin{lma}[multiplicative properties] \hspace{1mm} \label{additive}
\\ \\
Let $q \geq 0$.
\\ \\
a) For $s \in \mathbb{R}$ and $\textbf{{i}}, \textbf{{j}} \in \mathcal{I}^*$ we have

\begin{itemize}
\item[a1)]  If $s< \tau_1(q)+\tau_2(q)$, then $\psi^{s,q}({\textbf{{i}}}{\textbf{{j}}}) \leq  \psi^{s,q}({\textbf{{i}}})  \, \psi^{s,q}({\textbf{{j}}})$;

\item[a2)]  If $s= \tau_1(q)+\tau_2(q)$, then $\psi^{s,q}({\textbf{{i}}} {\textbf{{j}}}) =  \psi^{s,q}({\textbf{{i}}})  \, \psi^{s,q}({\textbf{{j}}})$;

\item[a3)]  If $s> \tau_1(q)+\tau_2(q)$, then $\psi^{s,q}({\textbf{{i}}} {\textbf{{j}}}) \geq  \psi^{s,q}({\textbf{{i}}})  \, \psi^{s,q}({\textbf{{j}}})$.

\end{itemize}

b) For $s \in \mathbb{R}$ and $k,l \in \mathbb{N}$ we have

\begin{itemize}
\item[b1)]  If $s< \tau_1(q)+\tau_2(q)$, then $\Psi_{k+l}^{s,q} \leq \Psi_{k}^{s,q} \, \Psi_{l}^{s,q}$;

\item[b2)]  If $s=\tau_1(q)+\tau_2(q)$, then $\Psi_{k+l}^{s,q} = \Psi_{k}^{s,q} \, \Psi_{l}^{s,q}$;

\item[b3)]  If $s> \tau_1(q)+\tau_2(q)$, then $\Psi_{k+l}^{s,q} \geq \Psi_{k}^{s,q} \, \Psi_{l}^{s,q}$.

\end{itemize}

\end{lma}

We will prove Lemma \ref{additive} in Section \ref{add}.  It follows from Lemma \ref{additive} and standard properties of sub- and super-multiplicative sequences that that we may define a function $P:\mathbb{R} \times [0, \infty) \to [0, \infty)$ by
\[
P(s,q) = \lim_{k \to \infty} (\Psi_k^{s,q}) ^{1/k} 
\]
where, in fact,
\[
\lim_{k \to \infty} (\Psi_k^{s,q}) ^{1/k} = \left\{ \begin{array}{cc}
\inf_{k \in \mathbb{N}} \,  (\Psi_k^{s,q}) ^{1/k}&  \text{if $s \in (-\infty,\tau_1(q)+\tau_2(q))$}\\ \\
\Psi_1^{s,q} &  \text{if $s  = \tau_1(q)+\tau_2(q)$}\\ \\
\sup_{k \in \mathbb{N}} \,  (\Psi_k^{s,q}) ^{1/k} &  \text{if $s \in (\tau_1(q)+\tau_2(q), \infty)$}
\end{array} \right. 
\]

Again, our function $P$ is a multifractal analogue of the function $P:\mathbb{R} \to [0,\infty)$ defined in \cite{me_box} and, in fact, $P(s,0) = P(s)$ for all $s \in \mathbb{R}$. Recall that $\tau_1$ and $\tau_2$ are Lipschitz continuous on $[\lambda, \infty)$ for any $\lambda>0$.  Let $L_\lambda>0$ be the larger of the two Lipschitz constants corresponding to $\tau_1$ and $\tau_2$ on $[\lambda, \infty)$.

\begin{lma}[Properties of $P$] \label{P} \hspace{1mm}
\begin{itemize}
\item[(1)]  For $s,r \in \mathbb{R}$ and $\lambda>0$, let
\[
U(s,r,\lambda) = \min\Big\{\alpha_{\min}^s \, p_{\min}^r, \alpha_{\min}^s \, p_{\max}^r, \alpha_{\max}^s \, p_{\min}^r, \alpha_{\max}^s \, p_{\max}^r  \Big\} \, \big(\alpha_{\max}/\alpha_{\min}\big)^{\min\{-L_\lambda r, 0\}}
\]
and
\[
V(s,r, \lambda) = \max\Big\{\alpha_{\min}^s \, p_{\min}^r, \alpha_{\min}^s \, p_{\max}^r, \alpha_{\max}^s \, p_{\min}^r, \alpha_{\max}^s \, p_{\max}^r\Big\} \, \big(\alpha_{\max}/\alpha_{\min}\big)^{\max\{-L_\lambda r, 0\}}.
\]
Then, for all $s,t \in \mathbb{R}$, $\lambda>0$, $q \geq \lambda$ and $r \geq \lambda-q$ we have
\[
U(s,r, \lambda)  P(t,q) \, \leq\,  P(s+t,q+r)\,  \leq \, V(s,r, \lambda)  P(t,q).
\]
Also, for all $s,t \in \mathbb{R}$, we have
\[
 \min\{\alpha_{\min}^s, \alpha_{\max}^s  \}  P(t,0) \, \leq\,  P(s+t,0)\,  \leq \, \max\{\alpha_{\min}^s, \alpha_{\max}^s  \}   P(t,0).
\]
Finally, for all $s \in \mathbb{R}$ and $q \geq 0$, we have
\[
P(s,q) \, \leq \, p_{\max}^q P(s,0).
\]
\item[(2)] $P$ is continuous on $\mathbb{R} \times (0, \infty)$ and on $\mathbb{R} \times \{0\}$;
\item[(3)] $P$ is strictly decreasing in $s$ and in $q$;
\item[(4)] For each $q \geq 0$, there is a unique value $s\geq 0$ for which $P(s,q)=1$.
\end{itemize}
\end{lma}

We will prove Lemma \ref{P} in Section \ref{Pproofs}.  It follows from Lemma \ref{P} (4) that we may define a function $\gamma: [0,\infty) \to \mathbb{R}$ by $P(\gamma(q),q)=1$.  This \emph{moment scaling function} is our main object of study.  Unfortunately, the definition for $\gamma(q)$ is not explicit, or even a closed form expression.  However, $\gamma(q)$ can be numerically estimated by approximating it by functions $\gamma_k$. For $k \in \mathbb{N}$ let $\gamma_k:[0,\infty) \to \mathbb{R}$ be defined by
\[
\Psi_k^{\gamma_k(q),q} = 1.
\]
The fact that this gives a well defined function $\gamma_k$ is easy to see.  
\begin{lma}[Properties of $\gamma_k$] \label{gammak} \hspace{1mm}
Let $k \in \mathbb{N}$.  We have
\begin{itemize}
\item[(1)] $\gamma_k$ is strictly decreasing on $[0,\infty)$;
\item[(2)] $\gamma_k$ is continuous on $(0,\infty)$;
\item[(3)] $\gamma_k(1) = 0$ and $\lim_{q \to \infty} \gamma(q) = -\infty$;
\item[(4)] $\gamma_k$ is convex on $(0,\infty)$.
\end{itemize}
\end{lma}

We will prove Lemma \ref{gammak} in Section \ref{gammakproofsnew}.

\begin{lma}[Properties of $\gamma$] \label{gamma} \hspace{1mm}
\begin{itemize}
\item[(1)] $\gamma$ is strictly decreasing on $[0,\infty)$;
\item[(2)] $\gamma$ is continuous on $(0,\infty)$;
\item[(3)] $\gamma$ is the pointwise limit of $\gamma_k$ as $k \to \infty$;
\item[(4)] $\gamma(1) = 0$ and $\lim_{q \to \infty} \gamma(q) = -\infty$;
\item[(5)] $\gamma$ is convex on $(0,\infty)$.
\end{itemize}
\end{lma}

We will prove Lemma \ref{gamma} in Section \ref{gammaproofs}.  One further key property of the $\gamma_k$ and $\gamma$ functions is differentiability.  This is more awkward to establish and, unsurprisingly, more important in terms of applications.  We will conduct a detailed study of this problem in Sections \ref{closedformsection}-\ref{diffingeneral}.

\subsection{$L^q$-spectra for our class of self-affine measures}

We can now state the main result od the paper, which in principle says that the $L^q$-spectrum of the self-affine measures introduced in Section \ref{boxdef} are equal to the moment scaling function defined in Section \ref{gammadefsection}.

\begin{thm} \label{main}
Let $\mu$ be a box-like self-affine measure.  Then
\begin{itemize}
\item[(1)] For $q \in [0,1]$, we have
\[
\overline{T}_{\text{C}, \mu}(q) = \overline{T}_{\P, \mu}(q)  = \overline{\tau}_\mu(q)  \leq \gamma(q);
\]
\item[(2)] For $q\geq1$, we have
\[
\gamma(q) \leq   \underline{T}_{\P, \mu}(q) = \underline{\tau}_\mu(q);
\]
\item[(3)] If, in addition, $\mu$ satisfies the ROSC, then,  for all $q \geq 0$, we have
\[
 T_{\P, \mu}(q) = \tau_\mu(q) = \gamma(q).
\]
\end{itemize}
\end{thm}

We will prove Theorem \ref{main} in Section \ref{mainproof}.  The following Corollary relates to the multifractal Hausdorff and packing spectra of $\mu$.
\begin{cor} \label{coro4444}
Let $\mu$ be a box-like self-affine measure which satisfies the ROSC.  Then
\[
f_{\text{\emph{H}},\mu}(\alpha) \leq f_{\text{\emph{P}},\mu}(\alpha) \leq \gamma^*(\alpha).
\]
\end{cor}

\begin{proof} See Proposition \ref{upperbounds}.
\end{proof}

We note that the upper bound for the multifractal spectra given in Corollary \ref{coro4444} is certainly not sharp in general.  Olsen \cite{sponges} demonstrated that, even in the much simpler Bedford-McMullen situation, $f_{\text{H},\mu}(\alpha) <\gamma^*(\alpha)$ is possible and more recently Reeve \cite[Theorem 7, Example 3]{reeve} and Jordan and Rams (personal communication) have shown that, also in the Bedford-McMullen setting, the packing spectrum exhibits many strange phenomenon and is in general not equal to $\gamma^*(\alpha)$, disproving a conjecture of Olsen, see \cite[Conjecture 4.1.7]{sponges}.

\subsection{Closed form expressions in the separated case}
\label{closedformsection}

In this section we use our main results to derive a very simple \emph{closed form expression} for $\gamma$ in the separated case.  These formulae are of particular interest in relation to \cite[Theorem 1]{fengaffine}, where the (negative of the) function $\gamma$ is expressed as a minimum of two expressions which are both infima over a simplex of probabilty vectors - in particular, not a closed form expression.  Since the class covered in \cite{fengaffine} is (strictly) contained in our separated class, this section provides very useful information if one is interested in explicit calculation. Moreover, we use the closed form expression to study the differentiability of the $L^q$-spectrum and the dimensions fo the measure and its support.
\\ \\
Assume $\mu$ is of separated type, which means that the linear part of each map $S_i$ in the defining IFS is of the form
\[
\left ( \begin{array}{cc}
\pm c_i &0\\ 
0 & \pm d_i\\
\end{array} \right )
\]
for constants $c_i, d_i \in (0,1)$, which are the singular values of $S_i$.  The $q$-modified singular value function is not necessarily multiplicative in this setting, but the functions
\[
p(\textbf{\emph{i}})^q \, c_\textbf{\emph{i}}^{\tau_1(q)} \, d_\textbf{\emph{i}}^{s-\tau_1(q)}
\]
and
\[
p(\textbf{\emph{i}})^q \,  d_\textbf{\emph{i}}^{\tau_2(q)} \, c_\textbf{\emph{i}}^{s-\tau_2(q)}
\]
\emph{are} multiplicative in $\textbf{\emph{i}}$.  This is not true in the non-separated case and is the key difference in the two settings.  Define functions $\gamma_A, \gamma_B: [0,\infty) \to \mathbb{R}$ by
\[
\sum_{i \in \mathcal{I}} p_i^q \, c_i^{\tau_1(q)} \, d_i^{\gamma_A(q)-\tau_1(q)} =1
\]
and
\[
\sum_{i \in \mathcal{I}} p_i^q \, d_i^{\tau_2(q)} \, c_i^{\gamma_B(q)-\tau_2(q)} =1
\]
respectively.

\begin{lma} \label{closeddiffs2}
If $\tau_1$ is differentiable at $q>0$, then $\gamma_A$ is differentiable at $q$, with
\[
\gamma'_A(q) =  - \frac{\sum_{i \in \mathcal{I}} p_i^q \, c_i^{ \tau_1(q)} \, \,  d_i^ {\gamma_A(q)-\tau_1(q)}\log \Big( p_i c_i^{\tau_1'(q)}   d_i^{- \tau_1'(q)} \Big)}{\sum_{i \in \mathcal{I}} p_i^q \,  c_i^{ \tau_1(q)} \, \,  d_i^ {\gamma_A(q)-\tau_1(q)}\log  d_i}
\]
and if $\tau_2$ is differentiable at $q>0$, then $\gamma_B$ is differentiable at $q$, with
\[
\gamma'_B(q) =  - \frac{\sum_{i \in \mathcal{I}} p_i^q \, d_i^{ \tau_2(q)} \, \,  c_i^ {\gamma_B(q)-\tau_2(q)}\log \Big( p_i d_i^{\tau_2'(q)}   c_i^{- \tau_2'(q)} \Big)}{\sum_{i \in \mathcal{I}} p_i^q \,  d_i^{ \tau_2(q)} \, \,  c_i^ {\gamma_B(q)-\tau_2(q)}\log  c_i}.
\]
\end{lma}
\begin{proof}
This follows immediately by implicit differentiation of the definitions of $\gamma_A$ and $\gamma_B$.
\end{proof}

\begin{lma} \label{notmiddle}
Let $q \geq 0$.  Either
\[
\max\{\gamma_A(q), \gamma_B(q)\} \leq \tau_1(q)+\tau_2(q)
\]
or
\[
\min\{\gamma_A(q), \gamma_B(q)\} \geq \tau_1(q)+\tau_2(q).
\]
Also, if $\tau_1$ and $\tau_2$ are differentiable at $1$, then either
\[
\max\{\gamma_A'(1), \gamma_B'(1)\} \leq \tau'_1(1)+\tau'_2(1)
\]
or
\[
\min\{\gamma'_A(1), \gamma'_B(1)\} \geq \tau'_1(1)+\tau'_2(1).
\]
\end{lma}
We will prove Lemma \ref{notmiddle} in Section \ref{notmiddleproof}.
\begin{thm} \label{closedform}
Let $\mu$ be of separated type and let $q \geq 0$. If $\max\{\gamma_A(q), \gamma_B(q)\} \leq \tau_1(q)+\tau_2(q)$, then
\[
\gamma(q) = \max\{\gamma_A(q), \gamma_B(q)\}.
\]
If $\min\{\gamma_A(q), \gamma_B(q)\} \geq \tau_1(q)+\tau_2(q)$, then
\[
\gamma(q) \leq \min\{\gamma_A(q), \gamma_B(q)\},
\]
with equality occurring if either of the following are satisfied:
\begin{itemize}
\item[(1)] $\sum_{i \in \mathcal{I}} p_i^q \, c_i^{\tau_1(q)} \, d_i^{\gamma_A(q)-\tau_1(q)} \, \log \big( c_i/d_i\big) \geq 0$
\item[(2)]  $\sum_{i \in \mathcal{I}} p_i^q \, d_i^{\tau_2(q)} \, c_i^{\gamma_B(q)-\tau_2(q)} \, \log \big( d_i/c_i\big) \geq 0$.
\end{itemize}
Moreover, if $c_i \geq d_i$ for all $i \in \mathcal{I}$, then $\gamma(q) = \gamma_A(q)$ for all $q\geq 0$, and if $d_i \geq c_i$ for all $i \in \mathcal{I}$, then $\gamma(q) = \gamma_B(q)$ for all $q\geq 0$, without any additional assumptions.
\end{thm}

We will prove Theorem \ref{closedform} in Section \ref{closedformproof}.  The final part concerning the case when either $c_i \geq d_i$ or $d_i \geq c_i$ for all $i \in \mathcal{I}$ was obtained in \cite[Theorem 2]{fengaffine} with the additional assumption that the rotational and reflectional parts of the maps were trivial.  Note that, provided
\[
\sum_{i \in \mathcal{I}} p_i \, \log \big( c_i/d_i\big) \neq  0,
\]
Theorem \ref{closedform} gives a precise formula for $\gamma(q)$ in a neighbourhood of $q=1$, because the expressions in (1) or (2) above are the negative of each other at $q=1$ and this condition guarantees that one of them is strictly greater than 0.  We can use Theorem \ref{closedform} to obtain more precise information about the differentiability of $\gamma$.

\begin{prop} \label{closedformdiff}
If $\mu$ is of separated type and $\tau_1$ and $\tau_2$ are differentiable at a point $q>0$ which is in an open interval where Theorem \ref{closedform} gives equality, then $\gamma$ is differentiable unless $q$ corresponds to a phase change in $\gamma$, i.e., $\gamma$ switches from $\gamma_A$ to $\gamma_B$ or vice versa and $\gamma_A'(q) \neq \gamma_B'(q)$ .
\\ \\
Moreover, if $c_i \geq d_i$ for all $i \in \mathcal{I}$ and $\tau_1$ is differentiable at $q>0$, then $\gamma$ is differentiable at $q$ with $\gamma'(q) = \gamma_A'(q)$, and if $d_i \geq c_i$ for all $i \in \mathcal{I}$ and $\tau_2$ is differentiable at $q>0$, then $\gamma$ is differentiable at $q$ with $\gamma'(q) = \gamma_B'(q)$.
\end{prop}

\begin{proof}
This follows immediately from Theorem \ref{closedform} and Lemma \ref{closeddiffs2}.
\end{proof}

In Section \ref{example2} we give an example of separated type for which the $L^q$-spectrum is not differentiable at one point due to a phase transition of the type described above.  Despite this we are able to prove differentiability at $q=1$, which is important for applications.
\begin{thm} \label{diffat1}
Let $\mu$ be of separated type and assume that $\tau_1$ and $\tau_2$ are differentiable at $q=1$.  Then $\gamma$ is differentiable at $q=1$ with
\[
\gamma'(1)= \left\{ \begin{array}{cc}
 \min \{\gamma_A'(1), \gamma_B'(1)\} &  \text{if $ \min \{\gamma_A'(1), \gamma_B'(1)\} \geq \tau'_1(1)+\tau'_2(1)$} \\ \\
 \max \{\gamma_A'(1), \gamma_B'(1)\} &  \text{if $ \max \{\gamma_A'(1), \gamma_B'(1)\}  \leq \tau'_1(1)+\tau'_2(1)$}
\end{array} \right. 
\]
\end{thm}
We will proof Theorem \ref{diffat1} in Section \ref{diffat1proof}.  Finally we present closed form expressions for the dimensions of $\mu$ and $F$ in the separated case.  Notably, we do not need the additional assumptions (1) or (2) made in Theorem \ref{closedform} to obtain the dimension results.
\begin{cor} \label{coro4}
Let $\mu$ be of separated type and assume it satisfies the ROSC.  Then
\[
\dim_{\text{\emph{B}}} F = \dim_{\text{\emph{P}}} F = \max\{\gamma_A(0),\gamma_B(0)\}.
\]
If, in addition, $\tau_1$ and $\tau_2$ are differentiable at $q=1$, then $\dim_{\text{\emph{H}}} \mu = \dim_{\text{\emph{P}}}  \mu = \dim_{\text{\emph{e}}} \mu = -\gamma'(1)$ which is equal to either $- \gamma_A'(1)$ or $-\gamma_B'(1)$ depending on the relative relationship with $\tau_1+\tau_2$, see Proposition \ref{diffat1}.  Both $- \gamma_A'(1)$ and $-\gamma_B'(1)$ have an explicit formula given by Lemma \ref{closeddiffs2}.
\\ \\
Moreover, if $c_i \geq d_i$ for all $i \in \mathcal{I}$, then $\dim_{\text{\emph{B}}} F = \dim_{\text{\emph{P}}} F =\gamma_A(0)$ and if $\tau_1$ is differentiable at $q=1$, then
\[
\dim_{\text{\emph{H}}} \mu  = \dim_{\text{\emph{P}}}  \mu = \dim_{\text{\emph{e}}} \mu= - \gamma_A'(1) = - \frac{\sum_{i \in \mathcal{I}} p_i \big( \log p_i + \tau_1'(1)\log( c_i/ d_i)\big) }{\sum_{i \in \mathcal{I}} p_i \log  d_i}.
\]
If $d_i \geq c_i$ for all $i \in \mathcal{I}$, then $\dim_{\text{\emph{B}}} F = \dim_{\text{\emph{P}}} F =\gamma_B(0)$ and if $\tau_2$ is differentiable at $q=1$, then
\[
\dim_{\text{\emph{H}}} \mu  = \dim_{\text{\emph{P}}}  \mu = \dim_{\text{\emph{e}}}  \mu = - \gamma_B'(1) = - \frac{\sum_{i \in \mathcal{I}} p_i \big(\log p_i+ \tau_2'(1)\log(  d_i/ c_i)\big) }{\sum_{i \in \mathcal{I}} p_i \log  c_i}.
\]
\end{cor}
\begin{proof}
This follows from Theorem \ref{closedform} and Proposition \ref{dimapps}.  For the box dimension result it is easily seen that $\max\{\gamma_A(0), \gamma_B(0)\} \leq \tau_1(0) + \tau_2(0)$ which guarantees that the dimension is given by the maximum and not the minimum.
\end{proof}
The final part of Corollary \ref{coro4} concerning the case when either $c_i \geq d_i$ or $d_i \geq c_i$ for all $i \in \mathcal{I}$ was obtained in \cite[Theorem 2]{fengaffine} with the additional assumption that the rotational and reflectional parts of the maps were trivial.  Slightly weaker versions of the box dimension result were obtained in \cite{me_box, baranski}.
\\ \\
Theorem \ref{closedform} only provides upper bounds in the case where $\min\{\gamma_A(q), \gamma_B(q)\} \geq \tau_1(q)+\tau_2(q)$ and neither (1) nor (2) are satisfied.  Initially we believed that this situation was vacuous, but eventually we were able to find an example where this occurred for a small range of values of $q$, although our numerical estimations still suggested that $\gamma(q) = \min\{\gamma_A(q), \gamma_B(q)\}$ in this range.
\begin{ques}
In the separated case, if $\min\{\gamma_A(q), \gamma_B(q)\} \geq \tau_1(q)+\tau_2(q)$ and neither (1) nor (2) are satisfied, is it still true that
\[
\gamma(q) = \min\{\gamma_A(q), \gamma_B(q)\}?
\]
\end{ques}
We conclude this section by noting that, even in this awkward setting, Theorem \ref{closedform} still provides useful computational information as
\[
\tau_1(q) +\tau_2(q) \leq \gamma_k(q) \leq \gamma(q) \leq \min\{\gamma_A(q), \gamma_B(q)\}
\]
for all $k \in \mathbb{N}$.  The fact that $\gamma_k(q)$ and thus $\gamma(q)$ are greater than or equal to $ \tau_1(q) +\tau_2(q)$ follows since
\[
\psi^{\tau_1(q) +\tau_2(q), q}(\textbf{\emph{i}}) \ = \ p_{\textbf{\emph{i}}}^q \,  c_{\textbf{\emph{i}}}^{\tau_1(q)} \,  d_{\textbf{\emph{i}}}^{\tau_2(q)}
\]
for each $\textbf{\emph{i}} \in \mathcal{I}^*$.

\subsection{On the differentiability of $\gamma$ in the non-separated case} \label{diffingeneral}

It follows immediately from convexity that $\gamma$ is differentiable on $(0,\infty)$ at all but countably many points and is semi-differentiable at every point.  However, identifying particular points where $\gamma$ is differentiable is awkward.  This is unfortunate as we are particularly interested in differentiability at $q=1$ due to the applications this has concerning the dimensions of $\mu$.
\\ \\
Fix $q>0$ and suppose that $\tau_1$ and $\tau_2$ are differentiable at $q$.  It follows easily that $\gamma_k$ is differentiable at $q$ for all $k \in \mathbb{N}$.  Moreover, implicit differentiation of $\Psi_k^{\gamma_k(q),q} = 1$ yields
\begin{equation} \label{diffeq1}
\sum_{\textbf{\emph{i}} \in \mathcal{I}^k} p(\textbf{\emph{i}})^q \, \alpha_1 (\textbf{\emph{i}})^{ \tau_\textbf{\emph{i}}(q)} \, \,  \alpha_2 (\textbf{\emph{i}})^ {\gamma_k(q)-\tau_\textbf{\emph{i}}(q)}\log  \Big( p(\textbf{\emph{i}}) \alpha_1 (\textbf{\emph{i}})^{\tau_{\textbf{\emph{i}}}'(q)} \alpha_2(\textbf{\emph{i}})^{\gamma_k'(q) -\tau_{\textbf{\emph{i}}}'(q)} \Big) = 0
\end{equation}
which upon solving for $\gamma_k'(q)$ gives
\begin{equation} \label{diffeq2}
\gamma_k'(q) = - \frac{\sum_{\textbf{\emph{i}} \in \mathcal{I}^k} p(\textbf{\emph{i}})^q \, \alpha_1 (\textbf{\emph{i}})^{ \tau_\textbf{\emph{i}}(q)} \, \,  \alpha_2 (\textbf{\emph{i}})^ {\gamma_k(q)-\tau_\textbf{\emph{i}}(q)}\log \Big( p(\textbf{\emph{i}}) \alpha_1 (\textbf{\emph{i}})^{\tau_{\textbf{\emph{i}}}'(q)}  \alpha_2(\textbf{\emph{i}})^{- \tau_{\textbf{\emph{i}}}'(q)} \Big)}{\sum_{\textbf{\emph{i}} \in \mathcal{I}^k} p(\textbf{\emph{i}})^q \, \alpha_1 (\textbf{\emph{i}})^{ \tau_\textbf{\emph{i}}(q)} \, \,  \alpha_2 (\textbf{\emph{i}})^ {\gamma_k(q)-\tau_\textbf{\emph{i}}(q)}\log \alpha_2 (\textbf{\emph{i}})}.
\end{equation}

One might hope to find situations where the sequence $ \gamma_k'(q)$ converges and $\gamma'(q) = \lim_{k \to \infty} \gamma_k'(q)$, but this is difficult to establish and indeed is not always true by virtue of the example in Section \ref{example2}.   We can prove that the $ \gamma_k'(q)$ converge in certain restricted circumstances (including $q=1$) using subadditivity. The expression (\ref{diffeq1}) demands further attention.  For $\textbf{\emph{i}} \in \mathcal{I}^*$ and $q>0$ define an alternative singular value function by
\[
\hat \psi^{s,q}(\textbf{\emph{i}}) = p(\textbf{\emph{i}}) \alpha_1 (\textbf{\emph{i}})^{\tau_{\textbf{\emph{i}}}'(q)} \alpha_2(\textbf{\emph{i}})^{s -\tau_{\textbf{\emph{i}}}'(q)}
\]
and let
\[
\hat \Psi_k^{s,q} = \sum_{\textbf{\emph{i}} \in \mathcal{I}^k}  \psi^{\gamma_k(q),q}(\textbf{\emph{i}}) \log \hat \psi^{s,q}(\textbf{\emph{i}}).
\]
Observe that the expression $\hat \Psi_k^{\gamma'_k(q),q} = 0$ recovers (\ref{diffeq1}).

\begin{lma}[multiplicativity and additivity properties] \label{additivediff}
Let $s \in \mathbb{R}$, $\textbf{i}, \textbf{j} \in \mathcal{I}^*$ and $k,l \in \mathbb{N}$.
\begin{itemize}
\item[(1)] If $s \leq \tau'_1(q) + \tau'_2(q)$, then
\[ 
\hat \psi^{s,q}(\textbf{i} \textbf{j}) \leq\hat \psi^{s,q}(\textbf{i}) \hat \psi^{s,q}(\textbf{j})
\]
\item[(2)]   If $s \geq \tau'_1(q) + \tau'_2(q)$, then
\[ 
\hat \psi^{s,q}(\textbf{i} \textbf{j}) \geq\hat \psi^{s,q}(\textbf{i}) \hat \psi^{s,q}(\textbf{j})
\]
\item[(3)] If $\gamma(q)  = \tau_1(q) + \tau_2(q)$ and $s \leq \tau'_1(q) + \tau'_2(q)$, then
\[ 
\hat \Psi_{k+l}^{s,q} \leq \hat \Psi_{k}^{s,q} + \hat \Psi_{l}^{s,q}
\]
\item[(4)]  If $\gamma(q) = \tau_1(q) + \tau_2(q)$ and $s \geq \tau'_1(q) + \tau'_2(q)$, then
\[ 
\hat \Psi_{k+l}^{s,q} \geq \hat \Psi_{k}^{s,q} + \hat \Psi_{l}^{s,q}
\]
\end{itemize}
\end{lma}

We will prove Lemma \ref{additivediff} in Section \ref{additivediffproof}.  Standard properties of subadditive sequences allow us to define a function $\hat P : \mathbb{R}  \times [0, \infty) \to \mathbb{R}$ by
\[
\hat P(s,q) = \lim_{k \to \infty} \tfrac{1}{k} \, \hat \Psi_k^{s,q}   = \left\{ \begin{array}{cc}
\inf_{k \in \mathbb{N}} \,  \tfrac{1}{k} \, \hat \Psi_k^{s,q} &  \text{if $s \leq \tau_1'(q)+\tau_2'(q)$}\\ \\
\sup_{k \in \mathbb{N}} \,   \tfrac{1}{k} \, \hat \Psi_k^{s,q}&  \text{if $s \geq \tau_1'(q)+\tau_2'(q)$}
\end{array} \right. 
\]
\begin{lma}[Properties of $\hat P$] \label{Phat}
Suppose that $\tau_1$ and $\tau_2$ are differentiable at $q$.
\begin{itemize}
\item[(1)]  For all $s,t \in \mathbb{R}$ we have
\[
\min\{ s \log \alpha_{\min},s \log \alpha_{\max} \}+ \hat P(t,q) \, \leq\,  \hat  P(s+t,q)\,  \leq \, \max\{ s \log \alpha_{\min},s \log \alpha_{\max} \}+ \hat P(t,q);
\]
\item[(2)] For each $q\geq 0$, $\hat P(s,q)$ is continuous in $s$;
\item[(3)] For each $q\geq 0$, $\hat P(s,q)$ is strictly decreasing in $s$;
\item[(4)] For each $q\geq 0$, there is a unique value $s \in\mathbb{R}$ such that $\hat P(s,q) = 0$;
\item[(5)] The unique value of $s \in\mathbb{R}$ given in (4) is equal to $\lim_{k \to \infty} \gamma_k'(q)$.
\end{itemize}
\end{lma}
\begin{cor} \label{gammadiffq=1}
Suppose that $\tau_1$ and $\tau_2$ are differentiable at $1$ .  Then $\gamma_k'(1)$ exist for all $k$ and converge as $k \to \infty$.
\end{cor}

This Corollary can be deduced since $\gamma(1) = \tau_1(1) = \tau_2(1)= 0$.  In the situations where $\gamma_k'(q)$ converges, to prove that the limit is $\gamma'(q)$, a possible strategy would be to establish equicontinuity of the family $\{\gamma_k'(q)\}$ on some interval $I$ containing $q$.  This seems awkward primarily because it cannot be true for all intervals $I$, even where $\tau_1$ and $\tau_2$ are differentiable.  If this was true one could apply the Arzel\`a-Ascoli Theorem to extract a convergent subsequence which would converge uniformly and thus converge to $\gamma'(q)$, proving that $\gamma$ is differentiable on any open interval where $\tau_1$ and $\tau_2$ are differentiable, but this is false, see Section \ref{example2}.

\begin{ques}
Is it true that if $\tau_1$ and $\tau_2$ are differentiable in a neighbourhood of 1, then there exists a sub-neighbourhood where the family $\{\gamma_k'(q)\}$ is equicontinuous? 
\end{ques}

We can at least use the work in this section to estimate the left and right derivatives of $\gamma$, which has applications in estimating the Hausdorff dimension of $\mu$.

\begin{thm} \label{generaldiffest}
Suppose that $\tau_1$ and $\tau_2$ are differentiable at $1$.
\begin{itemize}
\item[(1)] If $\gamma_1'(1) \leq \tau_1'(1)+\tau_2'(1)$, then
\[
\lim_{k \to \infty} \gamma_k'(1) = \inf_{k \in \mathbb{N}} \gamma_k'(1) \geq \gamma'_+(1) \geq \gamma'_-(1).
\]
\item[(2)] If $\gamma_1'(1) \geq \tau_1'(1)+\tau_2'(1)$, then
\[
\lim_{k \to \infty} \gamma_k'(1) = \sup_{k \in \mathbb{N}} \gamma_k'(1) \leq \gamma'_-(1) \leq \gamma'_+(1).
\]
\end{itemize}
These estimates give obvious estimates for the Hausdorff dimension of $\mu$ since $- \gamma'_+(1) \leq \dim_\text{\emph{H}} \mu \leq - \gamma'_-(1)$.
\end{thm}
We will prove Theorem \ref{generaldiffest} in Section \ref{generaldiffestproof}.

\subsection{The problem of negative $q$}
\label{negativeq}

It is perhaps unsatisfying that we cannot compute the $L^q$-spectrum for negative $q$.  This is a common problem caused by the fact that the standard inequalities either go in the wrong direction for negative $q$ or cause one to lose too much information to be of use.  In \cite[Theorem 4.1.3 (iii)]{sponges}, Olsen solved this problem for Bedford-McMullen carpets by computing $T_{P,\mu}(q)$ for all $q \in \mathbb{R}$ assuming a strong separation condition.  Even in the self-similar setting, the calculations for negative $q$ can be tricky.  Riedi \cite{riedi} studied this problem and observed that the computationally convenient `grid' definition of the $L^q$-spectrum introduced by Falconer, see \cite[Chapter 17]{falconer}, was unsatisfactory for negative $q$ and depended on the orientation of the grid.  (We will make use of Falconer's definition in this paper, see Section \ref{prelimsection}.)  Riedi's elegant solution was to introduce a new `grid definition', where instead of taking the measure of a square, you take the measure of the square enlarged by a uniform factor.  This avoids the problem of getting undesirably high estimates for squares with very small mass.  We would look to use Riedi's definition in our setting.  The first problem would be to consider existence problems in the (graph-directed) self-similar setting. 

\begin{ques} \label{question1}
Can Theorem \ref{gdexists} be extended to include negative $q$?  Specifically, do the $L^q$-spectra of (graph-directed) self-similar measures exist for all $q \in \mathbb{R}$?
\end{ques}

If the answer is `yes', then we can proceed as in the positive $q$ setting.  We can define a moment scaling function $\gamma$ on the whole of $\mathbb{R}$ and it seems straightforward to prove that this function would be an upper bound for the $L^q$-spectrum for negative $q$.  The other direction would be very awkward and perhaps we would need to introduce a very restrictive separation condition along the lines of \cite{sponges}, however, this $\gamma$  would be the obvious candidate.

\begin{ques} \label{question2}
If $\mu$ satisfies the ROSC, then for $q<0$ is it true that
\[
T_{\text{\emph{C}}, \mu}(q) = T_{\P, \mu}(q) = \tau_\mu(q) =  \gamma(q)
\]
where $\gamma$ is the function alluded to above?
\end{ques}

We conjecture, somewhat tentatively, that the answers to questions \ref{question1}-\ref{question2} are both in the affirmative, but acknowledge that the techniques used in this paper are insufficient to tackle the problem.

\section{Examples}

\subsection{An example with nontrivial rotations} \label{example1}

In order to illustrate our results we present an example and compute the $L^q$-spectrum.  We compare our example with the corresponding situation where all rotational and reflectional components are taken to be trivial.  Let $S_1,S_2,S_3: [0,1]^2 \to [0,1]^2$ be the affine maps which take $[0,1]^2$ to the 3 shaded rectangles on the left hand part of Figure 1, starting with the rectangle in the top middle and rotating clockwise.  Furthermore, suppose that the linear parts have been composed with: reflection in the vertical axis (top middle); clockwise rotation by 90 degrees (bottom right); and clockwise rotation by 270 degrees (bottom left).  With this IFS associate the probabilities $(1/5, 4/25,16/25)$ and let $\mu$ be the associated box-like self-affine measure, and let $E$ be the support of $\mu$.  Also, let $\nu$ be the corresponding self-affine measure where all rotational and reflectional components are taken to be trivial, and let $F$ be the support.  The unit square has been divided up into columns of widths $1/4, 1/2, 1/4$ and rows of heights $1/2, 1/2$.

\begin{figure}[H] 
	\centering
	\includegraphics[width=160mm]{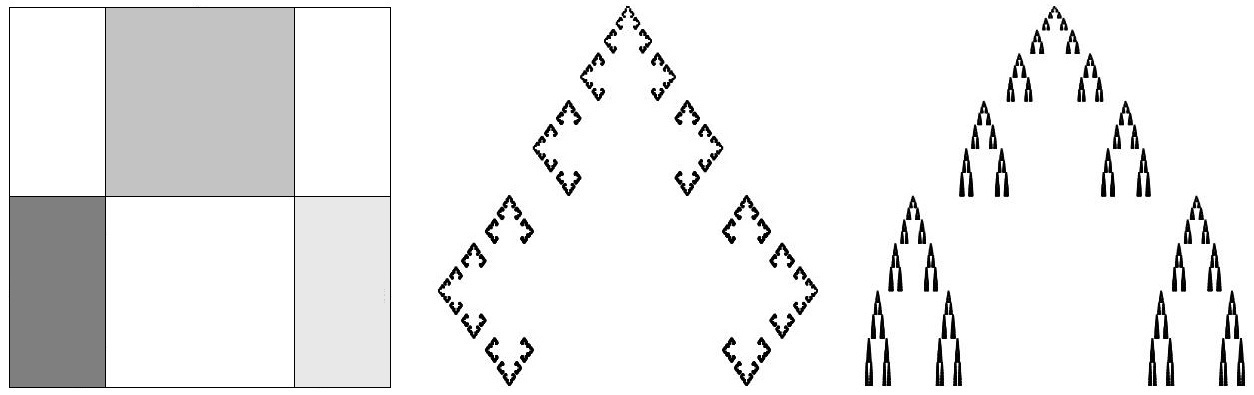}
\caption{The rectangles in the defining pattern for the above IFSs shaded according to measure (left); $E$, the self-affine support of $\mu$ (middle); and $F$, the self-affine support of $\nu$ (right). }
\end{figure}

Here, $\pi_1(\mu)$ and $\pi_2(\mu)$ are a pair of graph-directed self-similar measures defined by
\[
\pi_1(\mu)(A)= \tfrac{16}{25} \pi_2(\mu) (A/4)+\tfrac{1}{5} \pi_1(\mu) (-A/2+3/4)+\tfrac{4}{25} \pi_2(\mu) (A/4+3/4)
\]
and
\[
\pi_2(\mu)(A)= \tfrac{4}{5} \pi_1(\mu) (A/2)+\tfrac{1}{5} \pi_2(\mu) (A/2+1/2)
\]
for all Borel sets $A$.  Observe that the GDOSC is satisfied for this system and the associated weighted adjacency matrix is
\[
A^{(q,t)}=\left( \begin{array}{cc}
(\tfrac{1}{5})^q(\tfrac{1}{2})^t & (\tfrac{16}{25})^q(\tfrac{1}{4})^t+ (\tfrac{4}{25})^q(\tfrac{1}{4})^t\\ \\
(\tfrac{4}{5})^q(\tfrac{1}{2})^t  & (\tfrac{1}{5})^q(\tfrac{1}{2})^t  
 \end{array} \right).
\]
Define a function $\beta:\mathbb{R} \to \mathbb{R}$ by $\rho\big(A^{(\beta(q),q)}\big) = 1$ and a function $\gamma: \mathbb{R} \to \mathbb{R}$ by
\[
\lim_{k \to \infty} \bigg(\sum_{\textbf{\emph{i}} \in \mathcal{I}^k} \alpha_1(\textbf{\emph{i}})^{\beta(q)}\,  \alpha_2(\textbf{\emph{i}})^{\gamma(q)-\beta(q)}  \bigg)^{1/k} = 1.
\]
It follows from results in \cite{stric} that, since the measure separated open set condition holds, the $L^q$-spectra of $\pi_1(\mu)$ and $\pi_2(\mu)$ coincide with $\beta$ on $[0,\infty)$ and therefore Theorem \ref{main} gives that for all $q \in [0,\infty)$, $\tau_\mu(q) = \gamma(q)$.

\begin{figure}[H] 
	\centering
	\includegraphics[width=140mm]{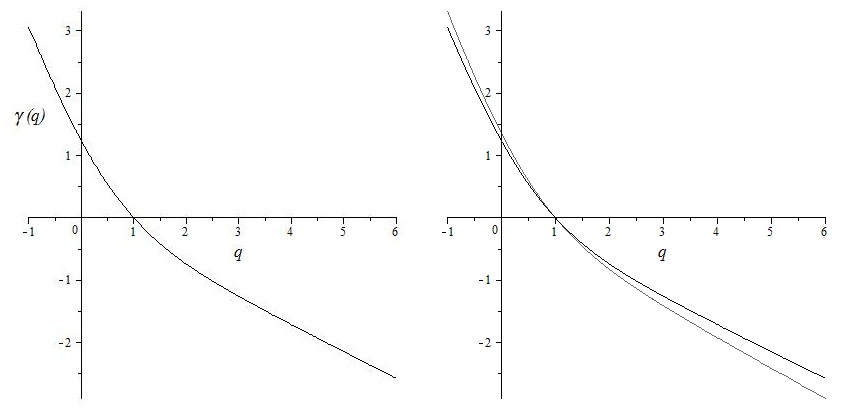}
\caption{Left: a graph of $\gamma(q)$. Right: for comparative purposes, graphs of the $L^q$-spectrum of $\mu$ (black) and $\nu$ (grey). }
\end{figure}
The parts of the above graphs corresponding to negative values of $q$ do not yet hold any geometric significance, but are simply our conjectured values for the $L^q$-spectrum in this range.  As discussed previously, the increasing parts of the Legendre transforms of the functions plotted above give upper bounds for the Hausdorff and packing multifractal spectra.
\begin{figure}[H] 
	\centering
	\includegraphics[width=145mm]{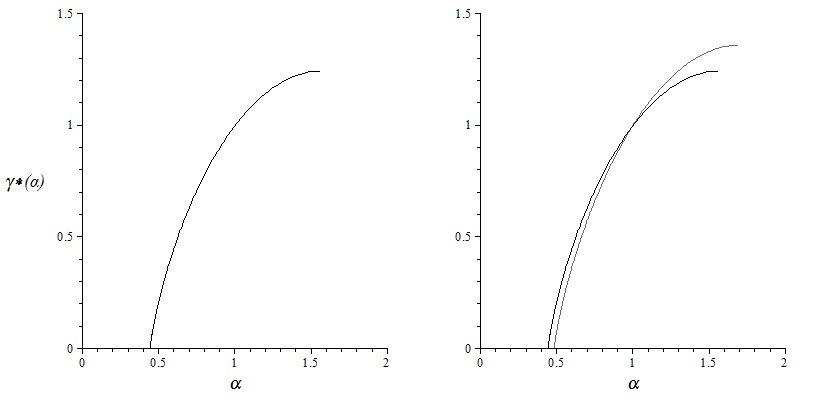}
\caption{Left: a graph of the increasing part of the Legendre spectrum of $\mu$, $\gamma^*(\alpha)$.  Right: for comparative purposes, graphs of the Legendre spectrum of $\mu$ (black) and $\nu$ (grey). }
\end{figure}
Finally, we compare the dimensions of $\mu$, $\nu$, $E$ and $F$.  Applying the results of Feng \cite{fengaffine} gives
\[
\dim_\text{P} F \, = \, 1.357018637, \qquad \text{and} \qquad \dim_\H  \nu \, = \, 1.042785026
\]
Estimating $\dim_\text{P} E$ from above, the 10th iterate gives
\[
\dim_\text{P} E \, \leq \,  \gamma_{10}(0) = 1.226824523 \, <  \, \dim_\text{P} F
\]
and, estimating $\dim_\H \mu$ from above using Theorem \ref{generaldiffest}, the 5th iterate gives
\[
\dim_\H \mu  \leq  -\gamma'_-(1) \leq - \gamma_5'(1) =   \,  0.9473061825 \, < \,  \dim_\H \nu.
\]
We conclude this section by remarking that the plots given here are only plots of \emph{approximations} to the desired functions.  However, we were able to graph high enough iterates such that the difference between successive estimates was indistinguishable to the naked eye in the given ranges of $q$.

\subsection{An example of a non-differentiable $L^q$-spectrum}  \label{example2}

In this section we present a simple example of a measure $\mu$ for which there is a point $q_0>0$ where the $L^q$-spectra are not differentiable.  This is in spite of the $L^q$-spectra of the projected measures, $\pi_1(\mu)$ and  $\pi_2(\mu)$, being differentiable for all $q >0$.  The self-affine measure falls into the class considered by Feng and Wang \cite{fengaffine} and is supported on a self-affine carpet of the type considered by Bara\'nski \cite{baranski}.  To the best of our knowledge this is the first example where the $L^q$-spectrum has been shown to be non-differentiable for a self-affine carpet.  In order to fully analyse this example we heavily rely on the fact that it is in the separated class and we thus have closed form expressions for all the quantities in question.
\\ \\
Let $\mu$ be the box-like self-affine carpet with support $F$ defined by the IFS depicted in the following figure.  All rotational and reflectional components are taken to be trivial.  The probabilities are taken to be $(3/5,1/5,1/5)$ and are indicated by shading as usual.  The unit square has been divided up into columns of widths $1/4, 1/2, 1/4$ and rows of heights $1/2, 3/10, 2/10$.

\begin{figure}[H] 
	\centering
	\includegraphics[width=120mm]{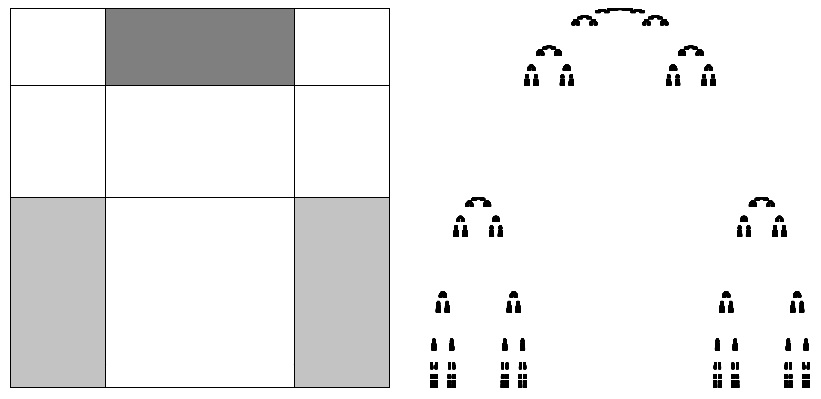}
\caption{Left: the defining pattern for the measure with the rectangles shaded according to mass. Right: the self-affine support. }
\end{figure}

A formula for the $L^q$-spectrum of $\mu$ is given in \cite[Theorem 1]{fengaffine}, but the formula is not a closed form expression and so computing it, plotting it, and analysing its differentiability are awkward.  The measures $\pi_1(\mu)$ and $\pi_2(\mu)$ are self-similar measures satisfying the measure separated open set condition and so their $L^q$-spectra can be computed using the standard closed formula and are differentiable for all $q \geq 0$.  Theorem \ref{closedform} and Theorem \ref{main} yield that the $L^q$-spectrum of $\mu$ is equal to either $\gamma_A$ or $\gamma_B$ depending on the relative relationship with $(\tau_1+\tau_2)$.  Thus we have a closed form expression for $\gamma(q) = \tau_\mu(q)$ and its derivative, where it exists.  It turns out that $\gamma$ has a phase transition at a point $q_0 \approx 0.237$, where it is not differentiable, but for all other values of $q \geq 0$ it is differentiable.  In fact, $\gamma(q) = \gamma_B(q)$ for $q \in [0,q_0]$ and $\gamma(q) = \gamma_A(q)$ for $q \in [q_0, -\infty)$ and the left derivative of $\gamma$ at $q_0$ is $\gamma_B'(q_0) = -1.160744186$ which is strictly less than the right hand derivative of $\gamma$ at $q_0$ which is $\gamma_B'(q_0) =  -1.010678931$.
\\ \\
We note that we are able to apply Theorem \ref{closedform} in its full strength because condition (1) is satisfied for all $q \geq 1$ and so in this region we get the \emph{equality} $\gamma(q) = \min\{\gamma_A(q), \gamma_B(q)\} = \gamma_A(q)$.

\begin{figure}[H] 
	\centering
	\includegraphics[width=140mm]{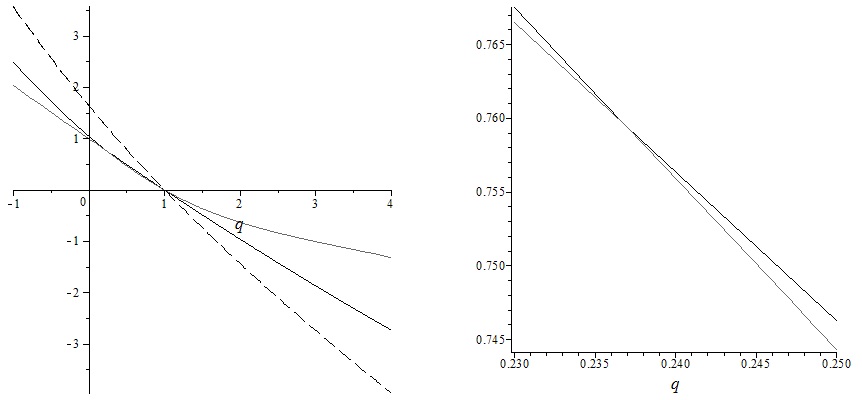}
\caption{Left: The graph of $\gamma$ (black), the graphs of the parts of $\gamma_A$ and $\gamma_B$ not equal to $\gamma$ (grey), and the graph of $(\tau_1+\tau_2)$ (dashed), which is included to indicate which of $\gamma_A, \, \gamma_B$ is equal to $\gamma$, i.e., the one `nearer' to  $(\tau_1+\tau_2)$.  Right: a magnification of the plot on the left to show the phase transition in better detail.}
\end{figure}

We can also easily compute the dimensions of $\mu$ and $F$ via closed form expressions:
\[
\dim_\text{B} F \, = \, \dim_\text{P} F \, = \, \gamma(0) \, = \, \gamma_B(0) \, = \,  1.046105401.
\]
and
\[
\dim_\H  \mu   \, =  \,  \dim_{\text{P}}  \mu \, =  \,  \dim_{\text{e}}  \mu  \, =  \,    -\gamma'(1)  \, =  \,  -\gamma_A'(1)  \, = \,  0.9792504246.
\]

\section{Preliminary results, notation and some inequalities} \label{prelimsection}

Let $\delta>0$ and let $\mathcal{D}_\delta$ be the set of closed cubes in a $\delta$ mesh imposed on $\mathbb{R}^d$ orientated at the origin.  Let
\[
D_\delta^q(\mu) = \sum_{Q \in \mathcal{ D}_\delta} \mu(Q)^q
\]
with the convention that $0^0 = 0$.  It turns out that for $q \geq 0$ the lower and upper $L^q$-spectrum of $\mu$ can be computed as follows
\[
\underline{\tau}_\mu(q) = \underline{\lim}_{\delta \to 0} \frac{\log D_\delta^q(\mu)}{- \log \delta}
\]
and
\[
\overline{\tau}_\mu(q) = \overline{\lim}_{\delta \to 0} \frac{\log D_\delta^q(\mu)}{-\log \delta}
\]
respectively, see \cite{falconer, exists}.  It is also worth noting that it is sufficient to only consider dyadic cubes, i.e. cubes with side lengths $2^{-n}$ with $n \to \infty$.  We will write $\hat D_n^q(\mu) = D_{2^{-n}}^q(\mu)$ and
$\mathcal{ \hat D}_n = \mathcal{D}_{2^{-n}}$ in an attempt to simplify notation when we use this fact.
\\ \\
We write $x \lesssim y$, if $x \leq C y$,  for some universal constant $C > 0$. Should we wish to emphasize that $C$ depends on some parameter $\theta$, we will write $x \lesssim_{\theta} y$.  If both $x \lesssim y$ and $x \gtrsim y$, then we will say $x$ and $y$ are \emph{comparable} and write $x \asymp y$.  We conclude this section with a simple lemma which is vital in the study of the $L^q$-spectrum.

\begin{lma} \label{inequality1}
Let $k \in \mathbb{N}$, $a_1, \dots, a_k \geq 0$ and $q\geq0$.  Then
\[
\Bigg( \sum_{i=1}^k  a_i \Bigg)^q \asymp_{k,q} \sum_{i=1}^k  a_i ^q.
\] 
\end{lma}

\begin{proof}
The proof follows immediately from variations of Jensen's inequality, keeping in mind the differences between the cases when $q \in [0,1)$ and $q>1$.  In particular, for $q \in [0,1]$ we trivially have
\[
\Bigg( \sum_{i=1}^k  a_i \Bigg)^q \leq  \ \sum_{i=1}^k  a_i ^q
\]
and, for $q \geq 1$,
\[
\Bigg( \sum_{i=1}^k  a_i \Bigg)^q \geq  \ \sum_{i=1}^k  a_i ^q.
\]
The other directions use Jensen's inequality for concave and convex functions respectively and introduce a constant depending on $k$ and $q$.
\end{proof}

The key reason the above lemma is important is that it allows us to freely move $q$ inside and outside expressions involving sums of measures.  If one wants to use the `difficult direction', namely the direction that requires Jensen's inequality, then one must be able to uniformly control the number $k$.  This is what the ROSC is used for in the subsequent proofs.  Also, the abject failure of the above lemma for negative $q$ goes a long way to explaining the problems in that setting.

\section{Proof of Theorem \ref{gdexists}} \label{gdproof}

In this section we will prove that the $L^q$-spectra exist for any family of graph-directed self-similar measures.  The proof follows \cite{exists} and could easily be extended to prove that $L^q$-spectra exist for any family of graph-directed self-\emph{conformal} measures but we only give the proof in the self-similar case in order to simplify exposition and focus on the key differences between the Peres-Solomyak argument and the graph-directed case.  Also, for the purposes of this paper, we only require the result in the self-similar setting.  We will prove that the $L^q$-spectrum exists for $\mu_1$ noting that the other arguments are symmetrical.  The fact that all the $L^q$-spectra coincide will follow from Lemma \ref{gdkeylem} bellow.
\\ \\
Let $\mathcal{E}_{1,j}^k$ denote all the paths of length $k$ in $\Gamma$ which start at the vertex $1$ and end on the vertex $j$ and let $\mathcal{E}_{1,j}^* = \bigcup_{k \in \mathbb{N}} \mathcal{E}_{1,j}^k$ and $\mathcal{E}_{1}^* = \bigcup_{j \in \mathcal{V}} \mathcal{E}_{1,j}^*$.  For $\textbf{\emph{e}} = (e_1, e_2, \dots, e_k) \in \mathcal{E}_{1}^*$ let
\[
S_\textbf{\emph{e}} = S_{e_1} \circ \cdots \circ S_{e_k}, \qquad c_\textbf{\emph{e}} = c_{e_1}  \cdots c_{e_k}, \qquad \text{and} \qquad  p_\textbf{\emph{e}} = p_{e_1}  \cdots p_{e_k}
\]
and let  $\overline{\textbf{\emph{e}}} = (e_1, e_2, \dots, e_{k-1}) \in \mathcal{E}_{1}^*$.  Consider the following $2^{-n}$ stopping
\[
{E}_{1,j}(n) = \{ \textbf{\emph{e}}  \in \mathcal{E}_{1,j}^* : c_{\textbf{\emph{e}}} \leq 2^{-n} < c_{\overline{\textbf{\emph{e}}}} \} 
\]
and observe that for all $n \in \mathbb{N}$
\begin{equation} \label{expformu}
\mu_1 = \sum_{j = 1}^{N} \sum_{\textbf{\emph{e}} \in \mathcal{E}_{1,j}(n)} p_\textbf{\emph{e}} \mu_j \circ S_\textbf{\emph{e}}^{-1}.
\end{equation}
We adopt the terminology used in \cite{exists} and say that a finite cover of $F_1$  by Borel sets $\{G_i\}_{i=1}^k$ is $(M, \varepsilon, N)$-good if $\lvert G_i \rvert \leq M \varepsilon$ for each $i$ and any cube of side length $\varepsilon$ intersects at most $N$ elements of the covering.  In particular, $(M, \varepsilon, N)$-good covers are efficient covers at scale $M\varepsilon$.

\begin{lma}[Lemma 2.2 in \cite{exists}] \label{basic1}
Let $j \in \mathcal{V}$ and $\{G_i\}_{i=1}^k$ be an $(M, 2^{-n}, N)$-good cover of $F_j$ and $q \geq 0$.  Then
\[
\hat D_n^q(\mu_j) \asymp \sum_{i=1}^k \mu_j(G_i)^q.
\]
\end{lma}

\begin{lma} \label{basic2}
There exists $M>0$ and $N \in \mathbb{N}$ such that for any $m,n \in \mathbb{N}$ any $j \in \mathcal{V}$ and any $\textbf{\emph{e}} \in  {E}_{j,i}(n) $, the collection
\[
\{S_\textbf{\emph{e}}^{-1} ( Q \cap F_j) : Q \in \mathcal{ \hat D}_{m+n}, Q \cap S_\textbf{\emph{e}}(F_j) \neq \emptyset\}
\]
is an $(M, 2^{-m}, N)$-good covering of $F_i$.
\end{lma}

\begin{proof}
This is a graph directed version of Lemma 2.4 in \cite{exists} and the proof follows easily in the same way.
\end{proof}

\begin{lma} \label{basic3}
For any $q\geq 0$, any $j, i \in \mathcal{V}$ and any $\textbf{\emph{e}} \in  {E}_{j,i}(n) $,
\[
\hat D_m^q(\mu_i) \asymp \sum_{Q \in \mathcal{ \hat D}_{m+n}} \mu_i(S_\textbf{\emph{e}}^{-1}Q)^q.
\]
\end{lma}

\begin{proof}
This follows immediately from Lemmas \ref{basic1}-\ref{basic2}.
\end{proof}

\begin{lma} \label{gdkeylem}
For any $q \geq 0$, $i,j \in \mathcal{V}$ and $m \in \mathbb{N}$, we have
\[
\hat D_m^q(\mu_i) \asymp \hat D_m^q(\mu_j).
\]
In particular, the ``$\asymp$'' does not depend on $m$.
\end{lma}

\begin{proof}
  Let $p_{\min} = \min_{e \in \mathcal{E}} p_e$ and fix $\textbf{\emph{e}} \in \mathcal{E}_{i,j}$.  We have
\begin{eqnarray*}
\hat D_m^q(\mu_i)  &=& \sum_{Q \in \mathcal{ \hat D}_m} \mu_i(Q)^q \\ \\
&\geq& \sum_{\substack{Q \in \mathcal{ \hat D}_m : \\ Q \cap S_\textbf{\emph{e}}(F_j) \neq \emptyset}} p_\textbf{\emph{e}}^q\mu_j \big(S_\textbf{\emph{e}}^{-1}(Q)\big)^q \\ \\
&\geq& p_{\min}^q \sum_{Q' \in S_\textbf{\emph{e}}^{-1}(\mathcal{ \hat D}_m)} \mu_j (Q')^q \\ \\
&\gtrsim& p_{\min}^q \sum_{Q' \in S_\textbf{\emph{e}}^{-1}(\mathcal{ \hat D}_m)}  \Bigg(\sum_{\substack{Q \in \mathcal{ \hat D}_m: \\ Q \cap Q' \neq \emptyset}}  \mu_j (Q \cap Q')\Bigg)^q.
\end{eqnarray*}
Observe that $\lvert \{ Q \in \mathcal{ \hat D}_m:  Q \cap Q' \neq \emptyset \} \rvert $ is bounded above by a universal constant depending only on the map $S_\textbf{\emph{e}}$.  Applying Lemma \ref{inequality1}, rearranging the pieces and applying Lemma \ref{inequality1} again in the other direction yields
\begin{eqnarray*}
\sum_{Q' \in S_\textbf{\emph{e}}^{-1}(\mathcal{ \hat D}_m)}  \Bigg(\sum_{\substack{Q \in \mathcal{ \hat D}_m: \\ Q \cap Q' \neq \emptyset}}  \mu_j (Q \cap Q')\Bigg)^q &\gtrsim&  \sum_{Q' \in S_\textbf{\emph{e}}^{-1}(\mathcal{ \hat D}_m)}  \sum_{\substack{Q \in \mathcal{ \hat D}_m: \\ Q \cap Q' \neq \emptyset}}  \mu_j (Q \cap Q')^q \\ \\
&=&   \sum_{ Q \in \mathcal{ \hat D}_m}  \sum_{\substack{Q' \in S_\textbf{\emph{e}}^{-1}(\mathcal{ \hat D}_m): \\ Q \cap Q' \neq \emptyset}}  \mu_j (Q \cap Q')^q \\ \\
&\gtrsim&   \sum_{ Q \in \mathcal{ \hat D}_m}  \Bigg(\sum_{\substack{Q' \in S_\textbf{\emph{e}}^{-1}(\mathcal{ \hat D}_m): \\ Q \cap Q' \neq \emptyset}}  \mu_j (Q \cap Q') \Bigg)^q \\ \\
&\geq& \sum_{ Q \in \mathcal{ \hat D}_m}  \mu_j (Q)^q \\ \\
&=& \hat D_m^q(\mu_j) 
\end{eqnarray*}
which yields $\hat D_m^q(\mu_i)  \gtrsim \hat D_m^q(\mu_j)$.  The other direction is symmetrical and so we have proved the lemma. 
\end{proof}

We can now proceed with the proof of Theorem \ref{gdexists}.  The structure and strategy closely follows \cite{exists}.

\begin{proof}
For $Q \in \mathcal{ \hat D}_{m+n}$, let $\tilde Q$ denote the unique member of $D_n$ such that $Q \subseteq \tilde Q$.
\\ \\
\textbf{Case 1:} $q \geq 1$.  Let
\[
p_+(\tilde Q) =  \sum_{j=1}^N  \sum_{\substack{\textbf{\emph{e}} \in \mathcal{E}_{1,j}(n) :\\ S_\textbf{\emph{e}}(F_j) \cap \tilde Q \neq \emptyset}} p_\textbf{\emph{e}}.
\]
We have
\begin{eqnarray*}
\hat D_{m+n}^q(\mu_1) &=& \sum_{\tilde Q \in \mathcal{ \hat D}_n} \sum_{Q \subseteq \tilde Q} \mu_1(Q)^q \\ \\
&=& \sum_{\tilde Q \in \mathcal{ \hat D}_n} \sum_{Q \subseteq \tilde Q} \Bigg( \sum_{j = 1}^{N} \sum_{\substack{\textbf{\emph{e}} \in \mathcal{E}_{1,j}(n) \\ S_\textbf{\emph{e}}(F_j) \cap \tilde Q \neq \emptyset}} p_\textbf{\emph{e}} \mu_j ( S_\textbf{\emph{e}}^{-1} Q ) \Bigg)^q \qquad \qquad \text{by (\ref{expformu})}\\ \\
&=& \sum_{\tilde Q \in \mathcal{ \hat D}_n} \sum_{Q \subseteq \tilde Q} p_+(\tilde Q) ^q \Bigg( \sum_{j = 1}^{N} \sum_{\substack{\textbf{\emph{e}} \in \mathcal{E}_{1,j}(n) \\ S_\textbf{\emph{e}}(F_j) \cap \tilde Q \neq \emptyset}} \frac{p_\textbf{\emph{e}}}{p_+(\tilde Q) } \mu_j ( S_\textbf{\emph{e}}^{-1} Q ) \Bigg)^q \\ \\
&\leq&   \sum_{\tilde Q \in \mathcal{ \hat D}_n} \sum_{Q \subseteq \tilde Q}  p_+(\tilde Q) ^q \sum_{j = 1}^{N} \sum_{\substack{\textbf{\emph{e}} \in \mathcal{E}_{1,j}(n) \\ S_\textbf{\emph{e}}(F_j) \cap \tilde Q \neq \emptyset}} \frac{p_\textbf{\emph{e}}}{p_+(\tilde Q) }  \mu_j ( S_\textbf{\emph{e}}^{-1} Q )^q \\ \\
&\,& \qquad \qquad \qquad \qquad \qquad \qquad \text{by Jensen's inequality for convex functions} \\ \\
&\leq&  \sum_{\tilde Q \in \mathcal{ \hat D}_n} p_+(\tilde Q) ^{q-1}  \sum_{j = 1}^{N} \sum_{\substack{\textbf{\emph{e}} \in \mathcal{E}_{1,j}(n) \\ S_\textbf{\emph{e}}(F_j) \cap \tilde Q \neq \emptyset}} p_\textbf{\emph{e}} \sum_{Q \in \mathcal{ \hat D}_{m+n}}  \mu_j ( S_\textbf{\emph{e}}^{-1} Q )^q \\ \\
&\lesssim&   \sum_{\tilde Q \in \mathcal{ \hat D}_n} p_+(\tilde Q) ^{q-1}  \sum_{j = 1}^{N} \sum_{\substack{\textbf{\emph{e}} \in \mathcal{E}_{1,j}(n) \\ S_\textbf{\emph{e}}(F_j) \cap \tilde Q \neq \emptyset}} p_\textbf{\emph{e}} \hat D_m^q(\mu_j) \qquad \qquad \text{by Lemma \ref{basic3}}\\ \\
&\lesssim&  \hat D_m^q(\mu_1) \sum_{\tilde Q \in \mathcal{ \hat D}_n} p_+(\tilde Q) ^{q-1}  \sum_{j = 1}^{N} \sum_{\substack{\textbf{\emph{e}} \in \mathcal{E}_{1,j}(n) \\ S_\textbf{\emph{e}}(F_j) \cap \tilde Q \neq \emptyset}} p_\textbf{\emph{e}}  \qquad \qquad \text{by Lemma \ref{gdkeylem}}\\ \\
&=&  \hat D_m^q(\mu_1)  \sum_{\tilde Q \in \mathcal{ \hat D}_n} p_+(\tilde Q) ^{q} \qquad \qquad \text{by the definition of $p_+(\tilde Q)$}
\end{eqnarray*}
and as in \cite{exists} it is easy to see that
\[
\sum_{\tilde Q \in \mathcal{ \hat D}_n} p_+(\tilde Q) ^{q} \lesssim \hat D_n^q(\mu_1)
\]
which yields
\[
\hat D_{m+n}^q(\mu_1) \lesssim \hat D_m^q(\mu_1) \hat D_n^q(\mu_1)
\]
which is the desired submultiplicativity condition and proves that $\log \hat D_{n}^q(\mu_1) /n \log 2$ converges for $q \geq 1$.
\\ \\
\textbf{Case 2:} $q \in [0,1]$.  For $\textbf{\emph{e}} \in \mathcal{E}_{1,j}(n)$ and $\tilde Q \in D_n$ let
\[
\omega_q(\textbf{\emph{e}}, \tilde Q) = \sum_{Q \in \mathcal{ \hat D}_{m+n}: Q \subseteq \tilde Q} \mu_j\big((S_\textbf{\emph{e}}^{-1}(Q) \big)^q
\]
For $\textbf{\emph{e}} \in \mathcal{E}_{1,j}(n)$ choose a box $\tilde Q$ which maximises $\omega_q(\textbf{\emph{e}}, \tilde Q)$ and denote it by $\tilde Q(\textbf{\emph{e}})$.  Write
\[
p_-(\tilde Q) =  \sum_{j=1}^N  \sum_{\substack{\textbf{\emph{e}} \in \mathcal{E}_{1,j}(n) : \\ \tilde Q(\textbf{\emph{e}}) =  \tilde Q }} p_\textbf{\emph{e}}.
\]
We have
\begin{eqnarray*}
\hat D_{m+n}^q(\mu_1) &=& \sum_{\tilde Q \in \mathcal{ \hat D}_n} \sum_{Q \subseteq \tilde Q} \mu_1(Q)^q \\ \\
&\geq& \sum_{\tilde Q \in \mathcal{ \hat D}_n} \sum_{Q \subseteq \tilde Q} \Bigg( \sum_{j = 1}^{N}  \sum_{\substack{\textbf{\emph{e}} \in \mathcal{E}_{1,j}(n) : \\ \tilde Q(\textbf{\emph{e}}) =  \tilde Q }}  p_\textbf{\emph{e}} \mu_j ( S_\textbf{\emph{e}}^{-1} Q ) \Bigg)^q \qquad \qquad \text{by (\ref{expformu})}\\ \\
&=& \sum_{\tilde Q \in \mathcal{ \hat D}_n} \sum_{Q \subseteq \tilde Q} p_-(\tilde Q)  ^q \Bigg( \sum_{j = 1}^{N}  \sum_{\substack{\textbf{\emph{e}} \in \mathcal{E}_{1,j}(n) : \\ \tilde Q(\textbf{\emph{e}}) =  \tilde Q }}  \frac{p_\textbf{\emph{e}}}{p_-(\tilde Q)  } \mu_j ( S_\textbf{\emph{e}}^{-1} Q ) \Bigg)^q \\ \\
&\geq&   \sum_{\tilde Q \in \mathcal{ \hat D}_n} \sum_{Q \subseteq \tilde Q}  p_-(\tilde Q)  ^q \sum_{j = 1}^{N}  \sum_{\substack{\textbf{\emph{e}} \in \mathcal{E}_{1,j}(n) : \\ \tilde Q(\textbf{\emph{e}}) =  \tilde Q }}  \frac{p_\textbf{\emph{e}}}{p_-(\tilde Q)  }  \mu_j ( S_\textbf{\emph{e}}^{-1} Q )^q \\ \\
&\,& \qquad \qquad \qquad \qquad \qquad \qquad \text{by Jensen's inequality for concave functions} \\ \\
&\gtrsim&  \sum_{\tilde Q \in \mathcal{ \hat D}_n} p_-(\tilde Q)  ^{q-1}  \sum_{j = 1}^{N}  \sum_{\substack{\textbf{\emph{e}} \in \mathcal{E}_{1,j}(n) : \\ \tilde Q(\textbf{\emph{e}}) =  \tilde Q }}  p_\textbf{\emph{e}} \sum_{Q \in \mathcal{ \hat D}_{m+n}}  \mu_j ( S_\textbf{\emph{e}}^{-1} Q )^q \\ \\
&\,&  \qquad \qquad \qquad \qquad \text{since we can bound the number of $\tilde Q \in \mathcal{\hat D}_n$ intersecting $S_\textbf{\emph{e}}(F_j)$ }\\ \\
&\gtrsim&   \sum_{\tilde Q \in \mathcal{ \hat D}_n} p_-(\tilde Q)  ^{q-1}  \sum_{j = 1}^{N}  \sum_{\substack{\textbf{\emph{e}} \in \mathcal{E}_{1,j}(n) : \\ \tilde Q(\textbf{\emph{e}}) =  \tilde Q }}  p_\textbf{\emph{e}} \hat D_m^q(\mu_j) \qquad \qquad \text{by Lemma \ref{basic3}}\\ \\
&\gtrsim& \hat D_m^q(\mu_1)   \sum_{\tilde Q \in \mathcal{ \hat D}_n} p_-(\tilde Q)  ^{q-1}  \sum_{j = 1}^{N}  \sum_{\substack{\textbf{\emph{e}} \in \mathcal{E}_{1,j}(n) : \\ \tilde Q(\textbf{\emph{e}}) =  \tilde Q }}  p_\textbf{\emph{e}} \qquad \qquad \text{by Lemma \ref{gdkeylem}}\\ \\
&=&  \hat D_m^q(\mu_1)   \sum_{\tilde Q \in \mathcal{ \hat D}_n} p_-(\tilde Q)  ^{q} \qquad \qquad \text{by the definition of $p_-(\tilde Q) $}
\end{eqnarray*}
and as in \cite{exists} it is easy to see that
\[
\sum_{\tilde Q \in \mathcal{ \hat D}_n} p_-(\tilde Q)  ^{q} \gtrsim \hat D_n^q(\mu_1)
\]
which yields
\[
\hat D_{m+n}^q(\mu_1) \gtrsim \hat D_m^q(\mu_1) \hat D_n^q(\mu_1)
\]
which is the desired supermultiplicativity condition and proves that $\log \hat D_{n}^q(\mu_1) /n \log 2$ converges for $q \in [0,1]$. 
\end{proof}

\section{Proofs of preliminary lemmas} \label{prelim}

\subsection{Proof of Lemma \ref{additive}} \label{add}

\emph{Proof of (a).}  Fix $q \geq 0$.
\\ \\
a1) Let $s \in \big[-\infty, \tau_1(q)+\tau_2(q)\big)$ and let $\textbf{\emph{i}},\textbf{\emph{j}} \in \mathcal{I}^*$.  Firstly, assume that $\mu$ is of non-separated type and hence $\tau_1(q) = \tau_2(q) = :t(q)$.  We have
\begin{eqnarray*}
\psi^{s,q}(\textbf{\emph{i}}\textbf{\emph{j}}) &=& p( \textbf{\emph{i}} \textbf{\emph{j}})^q \, \alpha_1( \textbf{\emph{i}} \textbf{\emph{j}})^{t(q)}\alpha_2( \textbf{\emph{i}} \textbf{\emph{j}})^{s-t(q)}\\ \\
&=&  p( \textbf{\emph{i}})^q p(\textbf{\emph{j}})^q \, \Big(\alpha_1( \textbf{\emph{i}} \textbf{\emph{j}}) \, \alpha_2( \textbf{\emph{i}} \textbf{\emph{j}}) \Big)^{s-t(q)} \, \alpha_1( \textbf{\emph{i}} \textbf{\emph{j}})^{2 t(q)-s}\\ \\
&=& p( \textbf{\emph{i}})^q p(\textbf{\emph{j}})^q \,  \Big(\alpha_1( \textbf{\emph{i}}) \, \alpha_2( \textbf{\emph{i}} ) \alpha_1(  \textbf{\emph{j}}) \, \alpha_2(  \textbf{\emph{j}}) \Big)^{s-t(q)} \, \alpha_1( \textbf{\emph{i}} \textbf{\emph{j}})^{2 t(q)-s} \\ \\
&\leq& p( \textbf{\emph{i}})^q p(\textbf{\emph{j}})^q \,  \Big(\alpha_1( \textbf{\emph{i}}) \, \alpha_2( \textbf{\emph{i}} )  \Big)^{s-t(q)}\, \Big( \alpha_1(  \textbf{\emph{j}})\alpha_2(  \textbf{\emph{j}}) \Big)^{s-t(q)} \, \Big(\alpha_1( \textbf{\emph{i}}) \, \alpha_1( \textbf{\emph{j}})\Big)^{2 t(q)-s} \qquad \text{since $2t(q)-s>0$}\\ \\
&=& \psi^{s,q}(\textbf{\emph{i}}) \, \psi^{s,q}( \textbf{\emph{j}})
\end{eqnarray*}
proving (a1) in the non-separated case.  Secondly, assume that $\mu$ is of separated type and assume, in addition, that $b(\textbf{\emph{i}}) \geq h(\textbf{\emph{i}})$, recalling that $b(\textbf{\emph{i}})$ and $h(\textbf{\emph{i}})$ are the lengths of the base and height of the rectangle $S_{\textbf{\emph{i}}}([0,1]^2)$ respectively.  The case where $b(\textbf{\emph{i}}) < h(\textbf{\emph{i}})$ is analogous.  Since $\mu$ is of separated type, $b(\textbf{\emph{ij}}) = b(\textbf{\emph{i}})\,b(\textbf{\emph{j}})$ and $h(\textbf{\emph{ij}}) = h(\textbf{\emph{i}})\,h(\textbf{\emph{j}})$ and this precludes the case: $b(\textbf{\emph{j}}) \geq h(\textbf{\emph{j}})$ and $b(\textbf{\emph{ij}}) < h(\textbf{\emph{ij}})$. We are left with the following three cases:
\begin{itemize}
\item[(i)] $b(\textbf{\emph{j}}) \geq h(\textbf{\emph{j}})$ and $b(\textbf{\emph{ij}}) \geq h(\textbf{\emph{ij}})$;
\item[(ii)] $b(\textbf{\emph{j}}) < h(\textbf{\emph{j}})$ and $b(\textbf{\emph{ij}}) \geq h(\textbf{\emph{ij}})$;
\item[(iii)] $b(\textbf{\emph{j}}) < h(\textbf{\emph{j}})$ and $b(\textbf{\emph{ij}}) < h(\textbf{\emph{ij}})$;
\end{itemize}
and in each situation we will show that
\[
\frac{\psi^{s,q}({\textbf{\emph{i}}}{\textbf{\emph{j}}})}{ \psi^{s,q}(\textbf{\emph{i}})  \, \psi^{s,q}({\textbf{\emph{j}}})} \leq 1.
\]
\textbf{Case (i):}
\[
\frac{\psi^{s,q}({\textbf{\emph{i}}}{\textbf{\emph{j}}})}{ \psi^{s,q}(\textbf{\emph{i}})  \, \psi^{s,q}({\textbf{\emph{j}}})}  = \frac{p(\textbf{\emph{ij}})^{q} b(\textbf{\emph{ij}})^{\tau_1(q)} h(\textbf{\emph{ij}})^{s-\tau_1(q)}}{p(\textbf{\emph{i}})^{q}p(\textbf{\emph{j}})^{q}b(\textbf{\emph{i}})^{\tau_1(q)} h(\textbf{\emph{i}})^{s-\tau_1(q)} b(\textbf{\emph{j}})^{\tau_1(q)} h(\textbf{\emph{j}})^{s-\tau_1(q)}}
= 1.
\]
\textbf{Case (ii):} 
\[
\frac{\psi^{s,q}({\textbf{\emph{i}}} {\textbf{\emph{j}}})}{ \psi^{s,q}(\textbf{\emph{i}})  \, \psi^{s,q}({\textbf{\emph{j}}})}  = \frac{p(\textbf{\emph{ij}})^{q} b(\textbf{\emph{ij}})^{\tau_1(q)} h(\textbf{\emph{ij}})^{s-\tau_1(q)}}{p(\textbf{\emph{i}})^{q} p(\textbf{\emph{j}})^{q}b(\textbf{\emph{i}})^{\tau_1(q)} h(\textbf{\emph{i}})^{s-\tau_1(q)} h(\textbf{\emph{j}})^{\tau_2(q)} b(\textbf{\emph{j}})^{s-\tau_2(q)}}
= \bigg(\frac{ b(\textbf{\emph{j}})}{h(\textbf{\emph{j}})} \bigg)^{\tau_1(q)+\tau_2(q)-s} \leq 1.
\]
\textbf{Case (iii):}
\[
\frac{\psi^{s,q}({\textbf{\emph{i}}} {\textbf{\emph{j}}})}{ \psi^{s,q}(\textbf{\emph{i}})  \, \psi^{s,q}({\textbf{\emph{j}}})}  = \frac{p(\textbf{\emph{ij}})^{q} h(\textbf{\emph{ij}})^{\tau_2(q)} b(\textbf{\emph{ij}})^{s-\tau_2(q)}}{p(\textbf{\emph{i}})^{q} p(\textbf{\emph{j}})^{q} b(\textbf{\emph{i}})^{\tau_1(q)} h(\textbf{\emph{i}})^{s-\tau_1(q)} h(\textbf{\emph{j}})^{\tau_2(q)} b(\textbf{\emph{j}})^{s-\tau_2(q)}}
= \bigg(\frac{ h(\textbf{\emph{i}})}{b(\textbf{\emph{i}})} \bigg)^{\tau_1(q)+\tau_2(q)-s} \leq 1.
\]
The proofs of (a2) and (a3) are similar and, therefore, omitted.
\\ \\
\emph{Proof of (b).}  This is straightforward by noting that, for all $k, l \in \mathbb{N}$, we have
\[
\Psi_{k+l}^{s,q} = \sum_{\textbf{\emph{i}} \in \mathcal{I}^{k+l}} \psi^{s,q} ({\textbf{\emph{i}}}) = \sum_{\textbf{\emph{i}} \in \mathcal{I}^{k}}  \sum_{\textbf{\emph{j}} \in \mathcal{I}^{l}}\psi^{s,q} ({\textbf{\emph{i}}} {\textbf{\emph{j}}}) 
\]
and
\[
\Psi_{k}^{s,q}  \, \Psi_{l}^{s,q}  = \Bigg(\sum_{\textbf{\emph{i}} \in \mathcal{I}^{k}} \psi^{s,q} ({\textbf{\emph{i}}})\Bigg) \Bigg( \sum_{\textbf{\emph{i}} \in \mathcal{I}^{l}} \psi^{s,q} ({\textbf{\emph{j}}}) \Bigg) = \sum_{\textbf{\emph{i}} \in \mathcal{I}^{k}}  \sum_{\textbf{\emph{j}} \in \mathcal{I}^{l}}\psi^{s,q} ({\textbf{\emph{i}}})  \, \psi^{s,q} ({\textbf{\emph{j}}})
\]
and applying part (a). \hfill \qed

\subsection{Proof of Lemma \ref{P}} \label{Pproofs}

(1)  Let $s,t \in \mathbb{R}$, $\lambda>0$, $q \geq \lambda$ and $r \geq \lambda-q$.  We have
\begin{eqnarray*}
  P(s+t, q+r) &=& \lim_{k \to \infty} \Bigg(\sum_{\textbf{\emph{i}} \in \mathcal{I}^{k}} p(\textbf{\emph{i}})^{q+r} \, \bigg(\frac{\alpha_1 (\textbf{\emph{i}})}{\alpha_2 (\textbf{\emph{i}})}\bigg)^{ \tau_\textbf{\emph{i}}(q+r)} \, \,  \alpha_2 (\textbf{\emph{i}})^ {s+t} \Bigg)^{1/k} \\ \\
&\leq& \lim_{k \to \infty} \Bigg(\max \big\{p_{\min}^{rk},  p_{\max}^{rk} \big\} \, \max \big\{\alpha_{\min}^{sk},  \alpha_{\max}^{sk} \big\} \sum_{\textbf{\emph{i}} \in \mathcal{I}^{k}} p(\textbf{\emph{i}})^{q} \, \bigg(\frac{\alpha_1 (\textbf{\emph{i}})}{\alpha_2 (\textbf{\emph{i}})}\bigg)^{ \tau_\textbf{\emph{i}}(q) +\max\{-L_\lambda r, 0\} } \, \,  \alpha_2 (\textbf{\emph{i}})^ {t} \Bigg)^{1/k} \\ \\
&\leq& V(s,r,\lambda) \, \lim_{k \to \infty} \Bigg( \sum_{\textbf{\emph{i}} \in \mathcal{I}^{k}} p(\textbf{\emph{i}})^{q} \, \bigg(\frac{\alpha_1 (\textbf{\emph{i}})}{\alpha_2 (\textbf{\emph{i}})}\bigg)^{ \tau_\textbf{\emph{i}}(q) } \, \,  \alpha_2 (\textbf{\emph{i}})^ {t} \Bigg)^{1/k} \\ \\
&=&  V(s,r,\lambda)  P(t,q).
\end{eqnarray*}
The proofs of the other inequalities are similar and omitted.
\\ \\
(2)  The continuity of $P(s,q)$  on $\mathbb{R} \times (0, \infty)$ and $\mathbb{R} \times \{0\}$ follows immediately from (1).
\\ \\
(3)  Let $s \in \mathbb{R}$, $q > 0$ and $\varepsilon > 0$.  Since $  P(s+\varepsilon, q),   P(s,q) \in (0, \infty)$, by (1) we have
\[
\frac{  P(s+\varepsilon, q)}{  P(s,q)} \leq V(\varepsilon, 0, q) = \alpha_{\max}^\varepsilon < 1
\]
and so $P(s,q)$ is strictly decreasing in $s$ for $q>0$.  Similarly
\[
\frac{  P(s, q+\varepsilon)}{  P(s,q)} \leq V(0,\varepsilon, q) = p_{\max}^\varepsilon < 1
\]
and so $P(s,q)$ is strictly decreasing in $q$ for $q>0$.  For the cases when $q=0$ the arguments are similar and omitted.
\\ \\
(4)  Fix $q \geq 0$.  It follows from (1) that $  P(s,q) >1$ as $s \to -\infty$ and that $  P(s, q)<1$ as $s \to \infty$.  These facts, combined with parts (2) and (3), imply that there is a unique value of $s$ for which $  P(s,q)=1$. \hfill \qed

\subsection{Proof of Lemma \ref{gammak}} \label{gammakproofsnew}

(1-3) follow immediately from the definition of $\gamma_k$, using the corresponding properties of $\tau_1$ and $\tau_2$.  We will now prove (4), which states that $\gamma_k$ is convex on $(0, \infty)$.  Let $k \in \mathbb{N}$, $0<q_0<q_1<\infty$ and let $\lambda >1$. We have
\begin{eqnarray*}
&\,& \hspace{-10mm} \Psi_k^{ \frac{\gamma_k(q_0)}{\lambda} + \frac{\gamma_k(q_1)(\lambda-1)}{\lambda}, \frac{q_0}{\lambda} + \frac{q_1(\lambda-1)}{\lambda}} \\ \\
&=& \sum_{\textbf{\emph{i}} \in \mathcal{I}^k}  p(\textbf{\emph{i}})^{\frac{q_0}{\lambda} + \frac{q_1(\lambda-1)}{\lambda}} \, \bigg(\frac{\alpha_1 (\textbf{\emph{i}})}{\alpha_2 (\textbf{\emph{i}})}\bigg)^{ \tau_\textbf{\emph{i}}\big(\frac{q_0}{\lambda} + \frac{q_1(\lambda-1)}{\lambda}\big) } \, \,  \alpha_2 (\textbf{\emph{i}})^ { \frac{\gamma_k(q_0)}{\lambda} + \frac{\gamma_k(q_1)(\lambda-1)}{\lambda}} \\ \\
&\leq& \sum_{\textbf{\emph{i}} \in \mathcal{I}^k}
 \Bigg( p(\textbf{\emph{i}})^{\frac{q_0}{\lambda} } \, \bigg(\frac{\alpha_1 (\textbf{\emph{i}})}{\alpha_2 (\textbf{\emph{i}})}\bigg)^{\frac{ \tau_\textbf{\emph{i}}(q_0)}{\lambda}  } \, \,  \alpha_2 (\textbf{\emph{i}})^ { \frac{\gamma_k(q_0)}{\lambda} }\Bigg) 
\Bigg( p(\textbf{\emph{i}})^{ \frac{q_1(\lambda-1)}{\lambda}} \, \bigg(\frac{\alpha_1 (\textbf{\emph{i}})}{\alpha_2 (\textbf{\emph{i}})}\bigg)^{ \frac{\tau_\textbf{\emph{i}}( q_1)(\lambda-1)}{\lambda} } \, \,  \alpha_2 (\textbf{\emph{i}})^ {  \frac{\gamma_k(q_1)(\lambda-1)}{\lambda}}\Bigg) \\ \\
&\,& \qquad \qquad  \qquad \qquad \qquad  \qquad \qquad  \qquad  \text{since $\tau_1$ and $\tau_2$ are convex on $(0,\infty)$}\\ \\ 
&\leq&  \Bigg( \sum_{\textbf{\emph{i}} \in \mathcal{I}^k}
p(\textbf{\emph{i}})^{q_0} \, \bigg(\frac{\alpha_1 (\textbf{\emph{i}})}{\alpha_2 (\textbf{\emph{i}})}\bigg)^{ \tau_\textbf{\emph{i}}(q_0)  } \, \,  \alpha_2 (\textbf{\emph{i}})^ { \gamma_k(q_0) }\Bigg)^{1/\lambda}
\Bigg( \sum_{\textbf{\emph{i}} \in \mathcal{I}^k} p(\textbf{\emph{i}})^{ q_1} \, \bigg(\frac{\alpha_1 (\textbf{\emph{i}})}{\alpha_2 (\textbf{\emph{i}})}\bigg)^{ \tau_\textbf{\emph{i}}( q_1) } \, \,  \alpha_2 (\textbf{\emph{i}})^ {  \gamma_k(q_1)}\Bigg) ^{(\lambda-1)/\lambda} \\ \\
&\,& \qquad \qquad  \qquad \qquad \qquad \qquad \qquad  \qquad  \text{by H\"older's inequality}\\ \\ 
&=&  \bigg(  \Psi_k^{ \gamma_k(q_0), q_0} \bigg)^{1/\lambda}
\bigg(  \Psi_k^{ \gamma_k(q_1), q_1} \bigg)^{(\lambda-1)/\lambda} \\ \\
&=&1
\end{eqnarray*}
by the definition of $\gamma_k$, which shows
\[
\gamma_k \Big(\frac{q_0}{\lambda} + \frac{q_1(\lambda-1)}{\lambda} \Big) \ \leq  \ \frac{\gamma_k(q_0)}{\lambda} + \frac{\gamma_k(q_1)(\lambda-1)}{\lambda}
\]
completing the proof.  \hfill \qed

\subsection{Proof of Lemma \ref{gamma}} \label{gammaproofs}

(1) The fact that $\gamma$ is strictly decreasing follows from the fact that $P(s,q)$ is strictly decreasing in both $s$ and $q$ and so in order to maintain $P(\gamma(q), q) = 1$, if $q$ increases, $\gamma(q)$ must decrease.
\\ \\
(2) Continuity of $\gamma$ follows easily from the continuity of $P$ and the identity $P(\gamma(q), q) = 1$.  Indeed, suppose $\gamma$ was not continuous.  Then we may find a sequence $q_n \to q$ such that $\gamma(q_n) \to s \neq \gamma(q)$.  However, $1 = P(\gamma(q_n), q_n) \to P(s,q) \neq 1$ which is a contradiction.
\\ \\
(3) This follows easily by the definition of $\gamma_k$ and the fact that $\Psi_k^{s,q} \to P(s,q)$ as $k \to \infty$. 
\\ \\
(4) This follows immediately from part (3) and Lemma \ref{gammak} (3).
\\ \\
(5) This follows immediately from part (3) and Lemma \ref{gammak} (4) since the pointwise limit of convex functions is convex.

\section{Proof of Theorem \ref{main}} \label{mainproof}

\subsection{Proofs of some key estimates} \label{mainproof0}

Let $q \geq 0$.  For
\[
\textbf{\emph{i}} = (i_1,i_2, \dots, i_{k-1}, i_k)  \in \mathcal{I}^*
\]
let
\[
 \overline{\textbf{\emph{i}}} = (i_1,i_2, \dots, i_{k-1}) \in \mathcal{I}^* \cup \{ \omega\},
\]
where $\omega$ is the empty word.  Note that the map $S_\omega$ is taken to be the identity map, which has singular values both equal to 1.  For $\delta \in (0,1]$ we define the $\delta$-\emph{stopping}, $\mathcal{I}_\delta$, as
\[
\mathcal{I}_\delta = \big\{\textbf{\emph{i}} \in \mathcal{I}^* : \alpha_2(\textbf{\emph{i}}) < \delta \leq \alpha_2( \overline{\textbf{\emph{i}}}) \big\}.
\]
Note that for $\textbf{\emph{i}} \in \mathcal{I}_\delta$ we have
\begin{equation} \label{stoppingest}
\alpha_{\min} \, \delta \leq \alpha_2(\textbf{\emph{i}}) < \delta.
\end{equation}

For $\textbf{\emph{i}} \in \mathcal{I}^*$, let $\mu_\textbf{\emph{i}} = p(\textbf{\emph{i}}) \,  \mu \circ S_\textbf{\emph{i}}^{-1}$ and $F_\textbf{\emph{i}} = S_\textbf{\emph{i}}(F) = \text{supp} \mu_\textbf{\emph{i}}$.  Note that for any $\delta \in (0,1]$, 
\[
\mu = \sum_{\textbf{\emph{i}} \in \mathcal{I}_\delta} \mu_\textbf{\emph{i}}.
\]
This fact will be used throughout the subsequent proofs without being mentioned explicitly.

\begin{lma} \label{deltaconv}
Let $t \in \mathbb{R}, q \geq 0$.

\begin{itemize}

\item[(1)]  If $t>\gamma(q)$, then
\[
\sum_{\textbf{{i}} \in \mathcal{I}_\delta}\psi^{t,q}(\textbf{{i}}) \  \lesssim_{t,q}  \ 1
\]
for all $\delta \in (0,1]$.

\item[(2)] If $t<\gamma(q)$, then
\[
\sum_{\textbf{{i}} \in \mathcal{I}_\delta}\psi^{t,q}(\textbf{{i}}) \  \gtrsim_{t,q} \ 1
\]
for all $\delta \in (0,1]$.

\end{itemize}
\end{lma}

\begin{proof}
(1)  Let $t>\gamma(q)$ and $\delta \in (0,1]$.  We have
\[
\sum_{\textbf{\emph{i}} \in \mathcal{I}_\delta} \psi^{t,q}(\textbf{\emph{i}}) \leq \sum_{\textbf{\emph{i}}\in \mathcal{I}^*}\psi^{t,q}(\textbf{\emph{i}})= \sum_{k=1}^\infty \sum_{\textbf{\emph{i}} \in \mathcal{I}^k}\psi^{t,q}(\textbf{\emph{i}}) = \sum_{k=1}^\infty \Psi_k^{t,q} < \infty
\]
since $\lim_{k \to \infty} (\Psi_k^{t,q})^{1/k} = P(t,q) < 1$.  The result follows since $\sum_{k=1}^\infty \Psi_k^{t,q}$ is a constant depending only on $t$ and $q$.
\\ \\
(2)  Let $t< \gamma(q)$.  Consider two cases according to whether $t$ is in the submultiplicative region $[0, \tau_1(q)+\tau_2(q)]$, or supermultiplicative region $(\tau_1(q)+\tau_2(q), \infty)$. 
\\ \\
\textbf{Case (i):} $0 \leq t \leq \tau_1(q)+\tau_2(q)$.  We remark that an argument similar to the following was used in \cite{affine, me_box}, but we include the details for completeness.  Let $\delta \in (0,1]$ and assume that
\begin{equation} \label{bounded}
\sum_{\textbf{\emph{i}} \in \mathcal{I}_\delta} \psi^{t,q}(\textbf{\emph{i}}) \leq 1.
\end{equation}
To obtain a contradiction we will show this implies that $t \geq \gamma(q)$.  Let $k(\delta) = \max \{ \lvert \textbf{\emph{i}} \rvert : \textbf{\emph{i}} \in \mathcal{I}_\delta \}$, where $\lvert \textbf{\emph{i}} \rvert$ denotes the length of the string $\textbf{\emph{i}}$, and let
\[
\mathcal{I}_{\delta, k} = \big\{\textbf{\emph{i}}_1 \dots \textbf{\emph{i}}_m : \textbf{\emph{i}}_j  \in \mathcal{I}_\delta \text{ for all $j = 1, \dots, m$}, \ \text{$\lvert \textbf{\emph{i}}_1 \dots \textbf{\emph{i}}_m\rvert \leq k$  but  $\lvert \textbf{\emph{i}}_1 \dots \textbf{\emph{i}}_{m} \textbf{\emph{i}}_{m+1}\rvert > k$  for some  $\textbf{\emph{i}}_{m+1} \in \mathcal{I}_\delta$}  \big\}.
\]
For all $\textbf{\emph{i}} \in \mathcal{I}^*$ we have, by the submultiplicativity of $\psi^{t,q}$,
\[
 \sum_{\textbf{\emph{j}} \in \mathcal{I}_\delta} \psi^{t,q}({\textbf{\emph{i}}  \textbf{\emph{j}}}) \ \leq \ \sum_{\textbf{\emph{j}} \in \mathcal{I}_\delta} \psi^{t,q}(\textbf{\emph{i}} ) \,  \psi^{t,q}(\textbf{\emph{j}})
 \ = \ \psi^{t,q}(\textbf{\emph{i}}) \, \sum_{\textbf{\emph{j}} \in \mathcal{I}_\delta}   \psi^{t,q}(\textbf{\emph{j}})
\ \leq \ \psi^{t,q}(\textbf{\emph{i}})
\]
by (\ref{bounded}).  It follows by repeated application of the above that, for all $k \in \mathbb{N}$,
\begin{equation} \label{bounded2old}
\sum_{\textbf{\emph{i}} \in \mathcal{I}_{\delta,k}} \psi^{t,q}(\textbf{\emph{i}}) \leq 1.
\end{equation}
Let $\textbf{\emph{i}} \in \mathcal{I}^{k}$ for some $k \in \mathbb{N}$.  It follows that $\textbf{\emph{i}} = \textbf{\emph{j}}_1 \, \textbf{\emph{j}}_2$ for some $\textbf{\emph{j}}_1 \in \mathcal{I}_{\delta,k}$ and some $\textbf{\emph{j}}_2 \in \mathcal{I}^* \cup \{\omega\}$ with $\lvert \textbf{\emph{j}}_2 \rvert \leq k(\delta)$ and by the submultiplicativity of $\psi^{t,q}$,
\[
\psi^{t,q}(\textbf{\emph{i}})\, = \,\psi^{t,q}({\textbf{\emph{j}}_1 \, \textbf{\emph{j}}_2}) \, \leq \,  \psi^{t,q}({\textbf{\emph{j}}_1}) \, \psi^{t,q}( {\textbf{\emph{j}}_2}) \, \leq \, c_{k(\delta)} \,  \psi^{t,q}({\textbf{\emph{j}}_1} ),
\]
where $c_{k(\delta)} = \max\{ \psi^{t,q}( {\textbf{\emph{i}}}) : \lvert \textbf{\emph{i}} \rvert \leq k(\delta) \}< \infty$ is a constant which depends only on $\delta$.  Since there are at most $\lvert \mathcal{I}\rvert^{k(\delta)+1}$ elements $\textbf{\emph{j}}_2 \in \mathcal{I}^*\cup \{\omega\}$ with $\lvert \textbf{\emph{j}}_2 \rvert \leq k(\delta)$ we have
\[
\Psi^{t,q}_{k} \, \, = \, \, \sum_{\textbf{\emph{i}} \in \mathcal{I}^{k}} \psi^{t,q}(\textbf{\emph{i}})  \, \, \leq  \, \, \lvert \mathcal{I} \rvert^{k(\delta)+1} \,c_{k(\delta)} \,  \sum_{\textbf{\emph{i}} \in \mathcal{I}_{\delta, k}} \psi^{t,q}(\textbf{\emph{i}})
 \, \, \leq  \, \, \lvert \mathcal{I} \rvert^{k(\delta)+1} \,c_{k(\delta)}
\]
by (\ref{bounded2old}).  Since this is true for all $k \in \mathbb{N}$ we have
\[
P(t,q) = \lim_{k \to \infty} \big(\Psi^{t,q}_{k}\big)^{1/k} \leq 1
\]
from which it follows that $t \geq s$.
\\ \\
\textbf{Case (ii):} $t >  \tau_1(q)+\tau_2(q)$.  Since $t<\gamma(q)$ it follows that $\sum_{\textbf{\emph{i}} \in \mathcal{I}^k} \psi^{t,q}(\textbf{\emph{i}}) \to \infty$ as $k \to \infty$.  Therefore, we may fix a $k \in \mathbb{N}$ such that
\begin{equation} \label{bigger1}
\sum_{\textbf{\emph{i}} \in \mathcal{I}^k} \psi^{t,q}(\textbf{\emph{i}}) \geq 1.
\end{equation}
Fix $\delta \in (0,1]$ and let
\begin{eqnarray*}
\mathcal{I}_{k,\delta} = \big\{\textbf{\emph{i}}_1 \dots \textbf{\emph{i}}_m &:&  \textbf{\emph{i}}_j  \in \mathcal{I}^k \text{ for all $j = 1, \dots, m$, } \quad \\ \\
&\,&\text{ $\alpha_2(\textbf{\emph{i}}_1 \dots \textbf{\emph{i}}_m) \geq \delta$ but $\alpha_2(\textbf{\emph{i}}_1 \dots \textbf{\emph{i}}_m \textbf{\emph{i}}_{m+1}) < \delta$ for some $\textbf{\emph{i}}_{m+1} \in \mathcal{I}^k$} \big\}.
\end{eqnarray*}
For all $\textbf{\emph{i}} \in \mathcal{I}^*$ we have, by the supermultiplicativity of $\psi^{t,q}$,
\[
 \sum_{\textbf{\emph{j}} \in \mathcal{I}^k} \psi^{t,q}({\textbf{\emph{i}}  \textbf{\emph{j}}})  \ \geq \  \sum_{\textbf{\emph{j}} \in \mathcal{I}^k} \psi^{t,q}(\textbf{\emph{i}}) \,  \psi^{t,q}(\textbf{\emph{j}}) \ = \  \psi^{t,q}(\textbf{\emph{i}} ) \, \sum_{\textbf{\emph{j}} \in \mathcal{I}^k}   \psi^{t,q}(\textbf{\emph{j}}) \ \geq \ \psi^{t,q}(\textbf{\emph{i}} )
\]
by (\ref{bigger1}).  It follows by repeated application of the above that
\begin{equation} \label{bounded2}
\sum_{\textbf{\emph{i}} \in \mathcal{I}_{k,\delta}} \psi^{t,q}(\textbf{\emph{i}}) \geq 1.
\end{equation}
Let $\textbf{\emph{i}} \in \mathcal{I}_{\delta}$.  It follows that $\textbf{\emph{i}} = \textbf{\emph{j}}_1 \textbf{\emph{j}}_2$ for some $\textbf{\emph{j}}_1 \in \mathcal{I}_{k,\delta}$ and some $\textbf{\emph{j}}_2 \in \mathcal{I}^*$.  Since $\alpha_2(\textbf{\emph{i}}) \geq \delta \, \alpha_{\min}$ by (\ref{stoppingest}) and $\alpha_2(\textbf{\emph{j}}_1) \leq \delta \alpha_{\min}^{-k}$ we have
\begin{equation} \label{lengthest}
\alpha_2(\textbf{\emph{j}}_1) \leq \alpha_2(\textbf{\emph{i}}) \alpha_{\min}^{-(k+1)} \leq \alpha_2(\textbf{\emph{j}}_1) \alpha_{\max}^{\lvert \textbf{\emph{j}}_2\rvert} \alpha_{\min}^{-(k+1)}
\end{equation}
which yields $\lvert \textbf{\emph{j}}_2 \rvert \leq (k+1)\frac{\log \alpha_{\min}}{\log \alpha_{\max}}$.  Setting $c_k  =  \min \Big\{ \psi^{t,q}(\textbf{\emph{i}}) : \lvert \textbf{\emph{i}}\rvert \leq (k+1) \frac{\log \alpha_{\min}}{\log \alpha_{\max}} \Big\} >0$, it follows from (\ref{lengthest}), (\ref{bounded2}) and the supermultiplicativity of $\psi^{t,q}$ that
\[
\sum_{\textbf{\emph{i}} \in \mathcal{I}_{\delta}} \psi^{t,q}(\textbf{\emph{i}}) \geq c_k  \sum_{\textbf{\emph{i}} \in \mathcal{I}_{k,\delta}}  \psi^{t,q}(\textbf{\emph{i}}) \geq c_k.
\]
Although $L(t,q)$ appears to depend on $k$, recall that we fixed $k$ at the beginning of the argument and the choice of $k$ depended only on $t$ and $q$.
\end{proof}

\begin{lma} \label{projectedmeasure}
For $q \geq 0$, $\delta \in (0, 1]$ and $\textbf{i} \in \mathcal{I}_\delta$ we have
\[
D_\delta^q \big(\mu_{\textbf{i}}\big) \ \asymp \   D_{\delta/\alpha_1(\textbf{i})}^q \big(p(\textbf{i}) \,\pi_{\textbf{i}}\mu \big).
\]
\end{lma}

\begin{proof}
This proof is straightforward and we only sketch it.  The key point is that since $\alpha_2(\textbf{\emph{i}}) \leq \delta$, the boxes which intersect $\text{supp}\mu_{\textbf{\emph{i}}}$ form a grid at most 3 deep in the direction of projection under $\pi_{\textbf{\emph{i}}}$.  This means that $D_\delta^q \big(\mu_{\textbf{\emph{i}}}\big)$ is comparable to $D_\delta^q $ of the projection of  $\mu_{\textbf{\emph{i}}}$ onto the longest side of the rectangle $S_ {\textbf{\emph{i}}}([0,1]^2)$ using Lemma \ref{inequality1}.  Finally, this measure is just a scaled down version of $p(\textbf{\emph{i}}) \,\pi_{\textbf{\emph{i}}}\mu$ by the factor $\alpha_1(\textbf{\emph{i}})$ and so scaling up completes the proof.
\end{proof}

\begin{lma} \label{simplebox}
For all $\varepsilon>0$, $\delta>0$, $q \geq 0$ and $p>0$  we have
\[
p^q \,  \delta^{-\tau_1(q) + \varepsilon/2} \ \lesssim_\varepsilon \  D_\delta^q(p \,\pi_1 \mu)  \ \lesssim_\varepsilon \ p^q \,  \delta^{-\tau_1(q) - \varepsilon/2}
\]
and
\[
p^q \, \delta^{-\tau_2(q) + \varepsilon/2} \  \lesssim_\varepsilon  \ D_\delta^q(p \,\pi_2 \mu) \  \lesssim_\varepsilon \ p^q \, \delta^{-\tau_2(q) - \varepsilon/2}.
\]
\end{lma}

\begin{proof}
This follows immediately from the definition of the $L^q$-spectrum.
\end{proof}

\subsection{Proof of Theorem \ref{main} (1)}

Let $q \in [0,1]$ and $\delta \in (0,1]$.  It suffices to show that $\overline{\tau}_{\mu}(q) \leq \gamma(q)$.   We have
\[
D_\delta^q(\mu) =\sum_{Q \in \mathcal{D}_\delta} \mu(Q)^q=  \sum_{Q \in \mathcal{D}_\delta} \Bigg(\sum_{\textbf{\emph{i}} \in \mathcal{I}_\delta} \mu_{\textbf{\emph{i}}} (Q)\Bigg)^q \leq  \sum_{Q \in \mathcal{D}_\delta} \sum_{\textbf{\emph{i}} \in \mathcal{I}_\delta} \mu_{\textbf{\emph{i}}} (Q)^q=\sum_{\textbf{\emph{i}} \in \mathcal{I}_\delta} \sum_{Q \in \mathcal{D}_\delta}  \mu_{\textbf{\emph{i}}} (Q)^q 
 =\sum_{\textbf{\emph{i}} \in \mathcal{I}_\delta}D_\delta^q(\mu_\textbf{\emph{i}})  .
\]
 It follows that, for all $\varepsilon>0$,
\begin{eqnarray*}
 \delta^{\gamma(q)+\varepsilon} D_\delta^q (\mu) &\leq& \delta^{\gamma(q)+\varepsilon} \sum_{\textbf{\emph{i}} \in \mathcal{I}_{\delta}} D_\delta^q \big(\mu_{\textbf{\emph{i}}}\big)\\ \\
&\lesssim&  \delta^{\gamma(q)+\varepsilon} \sum_{\textbf{\emph{i}} \in \mathcal{I}_{\delta}}  D_{\delta/\alpha_1(\textbf{\emph{i}})}^q \big(p(\textbf{\emph{i}}) \,\pi_{\textbf{\emph{i}}}\mu \big) \qquad \qquad   \text{by Lemma \ref{projectedmeasure}}\\ \\
&\ \lesssim_\varepsilon&   \delta^{\gamma(q)+\varepsilon} \sum_{\textbf{\emph{i}} \in \mathcal{I}_{\delta}} p(\textbf{\emph{i}})^q \,  \bigg(\frac{\delta}{\alpha_1(\textbf{\emph{i}})} \bigg)^{-\tau_{\textbf{\emph{i}}}(q)-\varepsilon/2} \qquad \qquad \text{by Lemma \ref{simplebox}}\\ \\
& \ \  \ \lesssim_{\varepsilon, q}&  \sum_{\textbf{\emph{i}} \in \mathcal{I}_{\delta}} p(\textbf{\emph{i}})^q \, \alpha_1(\textbf{\emph{i}})^{\tau_{\textbf{\emph{i}}}(q)+\varepsilon/2}\alpha_2(\textbf{\emph{i}})^{\gamma(q)+\varepsilon-\tau_{\textbf{\emph{i}}}(q)-\varepsilon/2}\qquad \qquad \text{by (\ref{stoppingest})}\\ \\
&=&  \sum_{\textbf{\emph{i}} \in \mathcal{I}_{\delta}} \psi^{\gamma(q)+\varepsilon/2, \, q}(\textbf{\emph{i}})\\ \\
&\ \ \ \lesssim_{\varepsilon, q}& 1
\end{eqnarray*}
by Lemma \ref{deltaconv} (1).  It follows that $\overline{\tau}_{\mu}(q) \leq \gamma(q)+\varepsilon$ and, since $\varepsilon>0$ was arbitrary, we have the desired upper bound. \hfill \qed

\subsection{Proof of Theorem \ref{main} (2)}

Let $q \geq 1$.  It suffices to show that $\underline{\tau}_{\mu}(q) \geq \gamma(q)$. We have
\[
D_\delta^q(\mu) =\sum_{Q \in \mathcal{D}_\delta} \mu(Q)^q
 =  \sum_{Q \in \mathcal{D}_\delta} \Bigg(\sum_{\textbf{\emph{i}} \in \mathcal{I}_\delta} \mu_{\textbf{\emph{i}}} (Q)\Bigg)^q \geq  \sum_{Q \in \mathcal{D}_\delta} \sum_{\textbf{\emph{i}} \in \mathcal{I}_\delta} \mu_{\textbf{\emph{i}}} (Q)^q =\sum_{\textbf{\emph{i}} \in \mathcal{I}_\delta} \sum_{Q \in \mathcal{D}_\delta}  \mu_{\textbf{\emph{i}}} (Q)^q=\sum_{\textbf{\emph{i}} \in \mathcal{I}_\delta}D_\delta^q(\mu_\textbf{\emph{i}}) .
\]
 It follows that, for all $\varepsilon>0$,
\begin{eqnarray*}
\delta^{\gamma(q)-\varepsilon} D_\delta^q (\mu) &\geq&   \delta^{\gamma(q)-\varepsilon} \sum_{\textbf{\emph{i}}\in \mathcal{I}_{\delta}} D_{\delta}^q \big(\mu_{\textbf{\emph{i}}}\big)\\ \\
&\gtrsim&\delta^{\gamma(q)-\varepsilon}    \sum_{\textbf{\emph{i}} \in \mathcal{I}_{\delta}} D_{\delta/\alpha_1(\textbf{\emph{i}})}^q\big(p(\textbf{\emph{i}}) \, \pi_{\textbf{\emph{i}}} \mu\big) \qquad \qquad  \text{by Lemma \ref{projectedmeasure}}\\ \\
& \ \gtrsim_\varepsilon& \delta^{\gamma(q)-\varepsilon} \sum_{\textbf{\emph{i}} \in \mathcal{I}_{\delta}}   p(\textbf{\emph{i}})^q \, \bigg(\frac{\delta}{\alpha_1(\textbf{\emph{i}})} \bigg)^{-\tau_{\textbf{\emph{i}}}(q)+\varepsilon/2} \qquad \qquad \text{by Lemma \ref{simplebox}} \\ \\
& \ \ \ \gtrsim_{\varepsilon, q}&  \sum_{\textbf{\emph{i}} \in \mathcal{I}_{\delta}}  p(\textbf{\emph{i}})^q \, \alpha_2(\textbf{\emph{i}})^{\gamma(q)-\varepsilon-\tau_{\textbf{\emph{i}}}(q)+\varepsilon/2} \, \alpha_1(\textbf{\emph{i}})^{\tau_{\textbf{\emph{i}}}(q)-\varepsilon/2}  \qquad \qquad \text{by (\ref{stoppingest})}\\ \\
&=&  \sum_{\textbf{\emph{i}} \in \mathcal{I}_{\delta}} \psi^{\gamma(q)-\varepsilon/2, \, q}(\textbf{\emph{i}})  \\ \\
& \ \ \ \gtrsim_{\varepsilon, q} & 1
\end{eqnarray*}
by Lemma \ref{deltaconv} (2).  It follows that $\underline{\tau}_{\mu}(q) \geq \gamma(q)-\varepsilon$ and, since $\varepsilon$ was arbitrary, we have the desired lower bound. \hfill $\qed$

\subsection{Proof of Theorem \ref{main} (3)}

Assume $\mu$ satisfies the ROSC.  In light of Theorem \ref{main} parts (1) and (2), to prove part (3) we only need to prove an upper bound in the case $q > 1$ and a lower bound in the case $q < 1$.
\\ \\
\textbf{Upper bound in the case $q > 1$.}  Examining the proof of the upper bound for $q \in [0,1]$, it is evident that the only place we needed the fact that $q \in [0,1]$ was to obtain
\[
D_\delta^q(\mu)  \leq  \sum_{\textbf{\emph{i}} \in \mathcal{I}_\delta}D_\delta^q(\mu_\textbf{\emph{i}}).
\]
For $q > 1$ we will use the ROSC to prove that
\[
D_\delta^q(\mu)  \lesssim  \sum_{\textbf{\emph{i}} \in \mathcal{I}_\delta}D_\delta^q(\mu_\textbf{\emph{i}})
\]
which is clearly sufficient to complete the proof.  Lemma \ref{inequality1} implies that
\[
\Bigg(\sum_{\textbf{\emph{i}} \in \mathcal{I}_\delta} \mu_{\textbf{\emph{i}}} (Q)\Bigg)^q \lesssim_{k,q} \sum_{\textbf{\emph{i}} \in \mathcal{I}_\delta} \mu_{\textbf{\emph{i}}} (Q)^q
\]
where
\[
k \ := \ \lvert \{\textbf{\emph{i}} \in \mathcal{I}_\delta : \mu_\textbf{\emph{i}}(Q) >0 \} \rvert.
\]
Thus, if we can uniformly bound $k$, for all $\delta$ and $Q \in \mathcal{D}_\delta$, then we are done.  Let $\delta \in (0,1]$ and $Q \in \mathcal{D}_\delta$.  Also, let $R$ be the open rectangle used in the ROSC and let $\theta$ denote the length of the shortest side of $R$.  Finally, let
\[
M = \min\big\{n \in \mathbb{N} : n \geq (\alpha_{\min}\theta)^{-1}+2\big\}.
\]
Since $\{S_\textbf{\emph{i}}(R)\}_{\textbf{\emph{i}}\in \mathcal{I}_\delta}$ is a collection of pairwise disjoint open rectangles each with shortest side having length at least $\alpha_{\min} \delta \theta$, it is clear that $D$ can intersect no more than $M^2$ of the sets $\{ F_{\textbf{\emph{i}}}\}_{\textbf{\emph{i}}\in \mathcal{I}_\delta}$.  Now since for each $\textbf{\emph{i}}$, $\text{supp} \mu_{\textbf{\emph{i}}} = F_{\textbf{\emph{i}}}$, it follows that $k \leq M^2$ completing the proof. \hfill $\qed$
\\ \\
\textbf{Lower bound in the case $q \in [0,1)$.}  Similar to above, finding a uniform bound for $k$ allows us to apply Lemma \ref{inequality1} to obtain
\[
\Bigg(\sum_{\textbf{\emph{i}} \in \mathcal{I}_\delta} \mu_{\textbf{\emph{i}}} (Q)\Bigg)^q \gtrsim_{k,q} \sum_{\textbf{\emph{i}} \in \mathcal{I}_\delta} \mu_{\textbf{\emph{i}}} (Q)^q
\]
and then the rest of the argument is identical to the $q \geq 1$ case. \hfill $\qed$

\section{Proofs concerning closed forms}

\subsection{Proof of Lemma \ref{notmiddle}} \label{notmiddleproof}
 Firstly, the situation
\[
\min\{\gamma_A(q), \gamma_B(q)\} < \tau_1(q)+\tau_2(q) < \max\{\gamma_A(q), \gamma_B(q)\}
\]
is not possible because
\[
 \sum_{i \in \mathcal{I}} p_i^q \, c_i^{\tau_1(q)} \, d_i^{(\tau_1(q)+\tau_2(q))-\tau_1(q)}   \ = \   \sum_{i \in \mathcal{I}} p_i^q \, c_i^{\tau_1(q)} \, d_i^{\tau_2(q)}   \ = \    \sum_{i \in \mathcal{I}} p_i^q \, d_i^{\tau_2(q)} \, c_i^{(\tau_1(q)+\tau_2(q))-\tau_2(q)} 
\]
and so if $\gamma_A(q) \leq \tau_1(q)+\tau_2(q)$, then so must $\gamma_B(q) \leq \tau_1(q)+\tau_2(q)$ with the reverse situation analogous. The differentials case is similar.  Note that $\gamma_A'(1)$ and $\gamma_B'(1)$ are given by the unique values of $s$ which make the following two expressions equal to zero respectively:
\[
\sum_{i \in \mathcal{I}} p_i \log \Big(p_i  c_i^{\tau'_1(1)} \, d_i^{s-\tau'_1(1)} \Big)
\]
and
\[
\sum_{i \in \mathcal{I}} p_i \log \Big( p_i \, d_i^{\tau'_2(1)} \, c_i^{s-\tau'_2(1)}  \Big).
\]
Since both expressions are strictly decreasing in $s$ and are both equal to
\[
\sum_{i \in \mathcal{I}} p_i \log \Big( p_i \, c_i^{\tau'_1(1)} \, d_i^{\tau'_2(1)} \Big)
\]
when evaluated at $s=\tau'_1(1)+\tau'_2(1)$ we deduce that if $\gamma_A'(1) \leq \tau_1'(1)+\tau_2'(1)$, then so must $\gamma_B'(1) \leq \tau_1'(1)+\tau_2'(1)$ with the reverse situation analogous. \hfill \qed

\subsection{Proof of Theorem \ref{closedform}} \label{closedformproof}

Let $\mu$ be of separated type and fix $q \geq 0$.  First let us deal with the case when $\max\{\gamma_A(q), \gamma_B(q)\} \leq \tau_1(q)+\tau_2(q)$ and assume without loss of generality that $\gamma_A(q) \leq \gamma_B(q)$.  Observe that for all $k \in \mathbb{N}$,
\[
1 \ = \  \sum_{i \in \mathcal{I}} p_i^q \, d_i^{\tau_2(q)} \, c_i^{\gamma_B(q)-\tau_2(q)}  \ = \  \sum_{\textbf{\emph{i}} \in \mathcal{I}^k} p(\textbf{\emph{i}})^q \, d_{\textbf{\emph{i}}}^{\tau_2(q)} \, c_{\textbf{\emph{i}}}^{\gamma_B(q)-\tau_2(q)}  \ \leq  \ \sum_{\textbf{\emph{i}} \in \mathcal{I}^k} \psi^{\gamma_B(q),q}( \textbf{\emph{i}} )
\]
which implies that $\gamma_B(q) \leq \gamma_k(q)$ and passing to the limit yields $\gamma_B(q) \leq \gamma(q)$.  For the reverse inequality let $\varepsilon>0$ and choose $k_0 \in \mathbb{N}$ such that for all $k \geq k_0$, $\gamma_k(q) \geq \gamma(q)-\varepsilon$.  For all $k \geq k_0$ we have
\begin{eqnarray*}
1 \ = \  \sum_{\textbf{\emph{i}} \in \mathcal{I}^k} \psi^{\gamma_k(q),q}( \textbf{\emph{i}} )&\leq& \sum_{\textbf{\emph{i}} \in \mathcal{I}^k} p(\textbf{\emph{i}})^q \, c_{\textbf{\emph{i}}}^{\tau_1(q)} \, d_{\textbf{\emph{i}}}^{\gamma_k(q)-\tau_1(q)} \ + \  \sum_{\textbf{\emph{i}} \in \mathcal{I}^k} p(\textbf{\emph{i}})^q \, d_{\textbf{\emph{i}}}^{\tau_2(q)} \, c_{\textbf{\emph{i}}}^{\gamma_k(q)-\tau_2(q)} \\ \\
&=& \Bigg( \sum_{i \in \mathcal{I}} p_i^q \, c_i^{\tau_1(q)} \, d_i^{\gamma_k(q)-\tau_1(q)} \Bigg)^k \ + \ \Bigg(\sum_{i \in \mathcal{I}} p_i^q \, d_i^{\tau_2(q)} \, c_i^{\gamma_k(q)-\tau_2(q)} \Bigg)^k\\ \\
&\leq& \Bigg( \sum_{i \in \mathcal{I}} p_i^q \, c_i^{\tau_1(q)} \, d_i^{\gamma(q)-\varepsilon-\tau_1(q)} \Bigg)^k \ + \ \Bigg(\sum_{i \in \mathcal{I}} p_i^q \, d_i^{\tau_2(q)} \, c_i^{\gamma(q)-\varepsilon-\tau_2(q)} \Bigg)^k.
\end{eqnarray*}
Since this is true for arbitrarily large $k$ and the expressions inside the large brackets in the last line of the above do not depend on $k$, at least one of them must be greater than or equal to $1$.  This yields $\gamma(q)-\varepsilon \leq \max\{\gamma_A(q), \gamma_B(q)\} = \gamma_B(q)$ and letting $\varepsilon$ tend to 0 completes the proof.
\\ \\
The situation where $\min\{\gamma_A(q), \gamma_B(q)\} \geq \tau_1(q)+\tau_2(q)$ is more challenging.  Assume temporarily that $\gamma_A(q) \leq \gamma_B(q)$ and observe that, similar to above, for all $k \in \mathbb{N}$,
\[
1 \ = \  \sum_{i \in \mathcal{I}} p_i^q \, c_i^{\tau_1(q)} \, d_i^{\gamma_A(q)-\tau_1(q)}  \ = \  \sum_{\textbf{\emph{i}} \in \mathcal{I}^k} p(\textbf{\emph{i}})^q \, c_{\textbf{\emph{i}}}^{\tau_1(q)} \, d_{\textbf{\emph{i}}}^{\gamma_A(q)-\tau_1(q)}  \ \geq  \ \sum_{\textbf{\emph{i}} \in \mathcal{I}^k} \psi^{\gamma_A(q),q}( \textbf{\emph{i}} )
\]
which implies that $\gamma_A(q) \geq \gamma_k(q)$ and passing to the limit yields $\gamma_A(q) \geq \gamma(q)$. The reverse inequality is considerably more difficult to handle and indeed we can only prove it if one of (1) or (2) in the statement of Theorem \ref{closedform} hold.   Assume that (1) is satisfied, i.e.,
\begin{equation} \label{Aassumption}
\sum_{i \in \mathcal{I}} p_i^q \, c_i^{\tau_1(q)} \, d_i^{\gamma_A(q)-\tau_1(q)} \, \log \big( c_i/d_i\big) \geq 0.
\end{equation}
We will prove that $\gamma(q) \geq \gamma_A(q) \geq \min\{\gamma_A(q), \gamma_B(q)\}$.  If (2) is satisfied, then the proof proceeds in an analogous fashion but using $\gamma_B$ instead of $\gamma_A$. We will use an `approximating from within' technique, somewhat inspired by \cite[Lemma 4.3]{fjs}.  The key is to find a subsystem which can `carry the pressure' and for which the singular value function is multiplicative.  We will use a version of Stirling's approximation for the logarithm of large factorials.  This states that for all $n \in \mathbb{N} \setminus \{1\}$ we have
\begin{equation} \label{stirling}
n \log n - n  \ \leq \ \log n! \ \leq \  n \log n - n +\log n.
\end{equation}
For $i \in \mathcal{I}$, let
\[
\theta_i = p_i c_i^{\tau_1(q)} d_i^{\gamma_A(q)-\tau_1(q)} \in (0,1)
\]
observing that
\[
\sum_{i \in \mathcal{I}} \theta_i \ = \ 1.
\]
For $k \in \mathbb{N}$, let
\[
n(k) = \sum_{i \in \mathcal{I}} \lfloor\theta_ik \rfloor \in \mathbb{N}
\]
and note that $k-\lvert \mathcal{I} \rvert \leq n(k) \leq k$.  Consider the $n(k)$th iteration of $\mathcal{I}$ and let
\[
\mathcal{J}_k = \Big\{ \textbf{\emph{j}} = (j_1, \dots, j_{n(k)}) \in \mathcal{I}^{n(k)} : \#\{m : j_m = i\} = \lfloor \theta_i k \rfloor \text{ for each $i \in \mathcal{I}$}\Big\}.
\]
It is straightforward to see that 
\begin{equation} \label{combina}
\lvert \mathcal{J}_k \rvert = \frac{n(k)!}{\prod_{i \in \mathcal{I}} \lfloor \theta_ik \rfloor !}
\end{equation}
is just a standard multinomial coefficient and for each $\textbf{\emph{j}} \in \mathcal{J}_k$ we have
\[
p_\textbf{\emph{j}} = \prod_{i \in \mathcal{I}} p_i^{\lfloor \theta_ik \rfloor} =: p,
\]
\[
c_\textbf{\emph{j}} = \prod_{i \in \mathcal{I}} c_i^{\lfloor \theta_ik  \rfloor} =: c
\]
and
\[
d_\textbf{\emph{j}} = \prod_{i \in \mathcal{I}} d_i^{\lfloor \theta_ik  \rfloor} =:  d.
\]
It follows from (\ref{Aassumption}) that $c \geq d$, from which we obtain that for all $\textbf{\emph{i}} \in \mathcal{J}_k$, we have
\[
\psi^{\gamma_A(q),q}(\textbf{\emph{i}}) \ = \   p^q \, c^{\tau_1(q)} \, d^{\gamma_A(q)-\tau_1(q)} \ = \   \prod_{i \in \mathcal{I}} \Bigg( \Big(p_i^{\lfloor \theta_ik \rfloor} \Big)^q \, \Big(c_i^{\lfloor \theta_ik \rfloor} \Big)^{\tau_1(q)} \, \Big(d_i^{\lfloor \theta_ik \rfloor} \Big)^{\gamma_A(q)-\tau_1(q)}  \Bigg) \ = \ \prod_{i \in \mathcal{I}} \theta_i^{{\lfloor \theta_ik \rfloor}}.
\]
This is the only part of the proof where we use (\ref{Aassumption}).  It follows that, for all $k > \max \{2 \max_i (\theta_i^{-1}), \lvert \mathcal{I} \rvert \}$,
\begin{eqnarray*}
\log \Big(\Psi_{n(k)}^{\gamma_A(q),q}\Big)^{1/n(k)} &\geq& \frac{1}{n(k)} \,  \log \Bigg(\sum_{\textbf{\emph{i}} \in \mathcal{J}_{k}} \psi^{\gamma_A(q),q}(\textbf{\emph{i}})  \Bigg) \\ \\
&=& \frac{1}{n(k)} \,  \log \Bigg(\lvert \mathcal{J}_{k} \rvert \ \prod_{i \in \mathcal{I}}  \theta_i^{{\lfloor \theta_ik \rfloor}} \Bigg) \\ \\
&=& \frac{1}{n(k)} \, \Bigg( \log n(k)!  -   \sum_{i \in \mathcal{I}} \log  \lfloor \theta_ik \rfloor !  +   \sum_{i \in \mathcal{I}} \lfloor \theta_ik \rfloor      \log  \theta_i \Bigg) \qquad \qquad  \text{by (\ref{combina})}\\ \\
&\geq& \frac{1}{n(k)} \, \Bigg( n(k) \log n(k) - n(k)  
- \sum_{i \in \mathcal{I}}  \lfloor \theta_ik \rfloor  \log  \lfloor \theta_ik \rfloor 
  +  \sum_{i \in \mathcal{I}}  \lfloor \theta_ik \rfloor  
- \sum_{i \in \mathcal{I}} \log  \lfloor \theta_ik \rfloor \\ \\
  &\,& \qquad \qquad \qquad  +   \sum_{i \in \mathcal{I}} \lfloor \theta_ik \rfloor      \log  \theta_i \Bigg)
\qquad \qquad  \text{by Stirling's approximation (\ref{stirling})} \\ \\
&\geq& \frac{1}{n(k)} \, \Bigg( n(k) \log n(k) 
- \sum_{i \in \mathcal{I}}  \lfloor \theta_ik \rfloor  \log   k  - \sum_{i \in \mathcal{I}}  \lfloor \theta_ik \rfloor  \log  \theta_i 
- \sum_{i \in \mathcal{I}} \log  \lfloor \theta_ik \rfloor +    \sum_{i \in \mathcal{I}} \lfloor \theta_ik \rfloor      \log  \theta_i \Bigg)\\ \\
&\geq& \frac{1}{n(k)} \, \Bigg( n(k) \log n(k) 
- n(k)  \log   k 
- \sum_{i \in \mathcal{I}} \log  \lfloor \theta_ik \rfloor \Bigg)\\ \\
&\geq& \log \bigg(\frac{k-\lvert \mathcal{I} \rvert}{k} \bigg)
- \frac{1}{k-\lvert \mathcal{I} \rvert}  \sum_{i \in \mathcal{I}} \log \theta_ik  \\ \\
&\to& 0
\end{eqnarray*}
as $k \to \infty$, which proves that
\[
P\big(\gamma_A(q),q \big) = \lim_{k \to \infty} \Big(\Psi_{n(k)}^{\gamma_A(q),q}\Big)^{1/n(k)} \geq 1
\]
which yields $\gamma(q) \geq \gamma_A(q)$ giving the result. Finally, we tie up the rest of the simple details in the Proposition.  If $c_i \geq d_i$ for all $i \in \mathcal{I}$, then
\[
\psi^{s,q}(\textbf{\emph{i}}) = p(\textbf{\emph{i}})^q \, c_{\textbf{\emph{i}}}^{\tau_1(q)} \, d_{\textbf{\emph{i}}}^{\gamma_k(q)-\tau_1(q)}
\]
for all $\textbf{\emph{i}} \in \mathcal{I}^*$, yielding $\gamma_k(q) = \gamma_A(q)$ for all $k$.  The situation where $d_i \geq c_i$ for all $i \in \mathcal{I}$ is similar and omitted.   \hfill \qed

\subsection{Proof of Theorem \ref{diffat1}} \label{diffat1proof}

First observe that $\gamma_A$, $\gamma_B$ and $(\tau_1+\tau_2)$ are all differentiable at 1 and therefore they must all have tangents at 1 intersecting at $(1,0)$ in the plane.  Also, $\gamma$ is convex and so must at least have left and right derivatives at $q=1$, with corresponding left and right tangents also meeting at $(0,1)$.  Since it is impossible for the tangent of $(\tau_1+\tau_2)$ to lie inbetween the other two tangents corresponding to $\gamma_A$ and $\gamma_B$ by Lemma \ref{notmiddle}, it must be either the steepest or the shallowest of the three. Assume we are in the first case, i.e. $ \min \{\gamma_A'(1), \gamma_B'(1)\} \geq \tau'_1(1)+\tau'_2(1)$,  and assume without loss of generality that $\min\{\gamma'_A(1), \gamma'_B(1)\}=\gamma'_A(1)$.  It follows that
\[
\tau_1(q)+\tau_2(q) \geq \gamma_A(q) \geq \gamma_B(q)
\]
for values of $q$ sufficiently close to, \emph{but less than}, 1, and so Theorem \ref{closedform} implies that $\gamma'_-(1) = \gamma_A(1)$. For values of $q$ sufficiently close to, \emph{but greater than}, 1, we have
\[
\tau_1(q)+\tau_2(q) \leq \gamma_A(q) \leq \gamma_B(q)
\]
and so Theorem \ref{closedform} tells us that $\gamma(q) \leq \gamma_A(q)$.  This implies that the right tangent of $\gamma$ lies inbetween the tangent of $(\tau_1+\tau_2)$ and the tangent of $\gamma_A$, but convexity guarantees that the right derivative of $\gamma$ is greater than or equal to the left derivative which is equal to $\gamma_A'(1)$.  Hence $\gamma'_+(1) = \gamma_A(1)$. The second case is the same apart from that we determine the right derivative first and then use convexity to prove equality with the left derivative. \hfill \qed

\section{Proofs concerning convergence of derivatives}

\subsection{Proof of Lemma \ref{additivediff}} \label{additivediffproof}

Parts (1) and (2) can be proved in an almost identical way to Lemma \ref{additive} (a) and so we omit the details.  We will now prove part (3).  The proof of part (4) is similar and also omitted.
\\ \\
(3)  Suppose $\gamma(q) = \tau_1(q) + \tau_2(q)$ and $s \leq \tau'_1(q) + \tau'_2(q)$.  It follows that $\gamma(q) = \gamma_k(q)$ for all $k \in \mathbb{N}$ and hence
\begin{eqnarray*}
\hat \Psi_{k+l}^{s,q} &=& \sum_{\textbf{\emph{i}} \in \mathcal{I}^{k+l}}  \psi^{\gamma_{k+l}(q),q}(\textbf{\emph{i}}) \log \hat \psi^{s,q}(\textbf{\emph{i}}) \\ \\
 &=& \sum_{\textbf{\emph{i}} \in \mathcal{I}^{k}}\sum_{\textbf{\emph{j}} \in \mathcal{I}^{l}}  \psi^{\gamma(q),q}(\textbf{\emph{i}} \textbf{\emph{j}}) \log \hat \psi^{s,q}(\textbf{\emph{i}}\textbf{\emph{j}}) \\ \\
 &\leq& \sum_{\textbf{\emph{i}} \in \mathcal{I}^{k}}\sum_{\textbf{\emph{j}} \in \mathcal{I}^{l}}  \psi^{\gamma(q),q}(\textbf{\emph{i}}) \psi^{\gamma(q),q}( \textbf{\emph{j}}) \Big( \log \hat \psi^{s,q}(\textbf{\emph{i}}) + \log \hat \psi^{s,q}(\textbf{\emph{j}}) \Big)  \\ \\
 &=& \sum_{\textbf{\emph{i}} \in \mathcal{I}^{k}} \psi^{\gamma(q),q}(\textbf{\emph{i}})  \log \hat \psi^{s,q}(\textbf{\emph{i}}) \sum_{\textbf{\emph{j}} \in \mathcal{I}^{l}}  \psi^{\gamma(q),q}(\textbf{\emph{j}})
\ + \
\sum_{\textbf{\emph{i}} \in \mathcal{I}^{k}} \psi^{\gamma(q),q}(\textbf{\emph{i}})  \sum_{\textbf{\emph{j}} \in \mathcal{I}^{l}}  \psi^{\gamma(q),q}(\textbf{\emph{j}})  \log \hat \psi^{s,q}(\textbf{\emph{j}})\\ \\
&=& \sum_{\textbf{\emph{i}} \in \mathcal{I}^{k}} \psi^{\gamma_k(q),q}(\textbf{\emph{i}})  \log \hat \psi^{s,q}(\textbf{\emph{i}}) 
\ + \
  \sum_{\textbf{\emph{j}} \in \mathcal{I}^{l}}  \psi^{\gamma_l(q),q}(\textbf{\emph{j}})  \log \hat \psi^{s,q}(\textbf{\emph{j}})\\ \\
&=& \hat \Psi_{k}^{s,q} + \hat \Psi_{l}^{s,q}
\end{eqnarray*}
completing the proof. \hfill \qed

\subsection{Proof of Lemma \ref{Phat}} \label{Phatproofs}

(1)  Let $s, t \in \mathbb{R}$.  We have
\begin{eqnarray*}
\hat  P(s+t, q) &=& \lim_{k \to \infty} \frac{1}{k} \sum_{\textbf{\emph{i}} \in \mathcal{I}^k}  \psi^{\gamma_{k}(q),q}  (\textbf{\emph{i}}) \log  \Big( p(\textbf{\emph{i}}) \alpha_1 (\textbf{\emph{i}})^{\tau_{\textbf{\emph{i}}}'(q)} \alpha_2(\textbf{\emph{i}})^{s+t -\tau_{\textbf{\emph{i}}}'(q)} \Big) \\ \\
&\leq&  \lim_{k \to \infty} \frac{1}{k} \sum_{\textbf{\emph{i}} \in \mathcal{I}^k} \psi^{\gamma_{k}(q),q}(\textbf{\emph{i}}) \log  \Big( \max\{\alpha_{ \min}^{ks},\alpha_{ \max}^{ks}\}  p(\textbf{\emph{i}}) \alpha_1 (\textbf{\emph{i}})^{\tau_{\textbf{\emph{i}}}'(q)} \alpha_2(\textbf{\emph{i}})^{t -\tau_{\textbf{\emph{i}}}'(q)} \Big) \\ \\
&=&  \lim_{k \to \infty} \frac{1}{k} \sum_{\textbf{\emph{i}} \in \mathcal{I}^k}  \psi^{\gamma_{k}(q),q}(\textbf{\emph{i}}) \log  \Big(  p(\textbf{\emph{i}}) \alpha_1 (\textbf{\emph{i}})^{\tau_{\textbf{\emph{i}}}'(q)} \alpha_2(\textbf{\emph{i}})^{t -\tau_{\textbf{\emph{i}}}'(q)} \Big) \\ \\
&\,& \qquad \qquad \qquad  \qquad \qquad \qquad + \max\{s \log \alpha_{ \min}, s \log \alpha_{ \max}\} \lim_{k \to \infty} \frac{1}{k} \sum_{\textbf{\emph{i}} \in \mathcal{I}^k}  \psi^{\gamma_{k}(q),q}(\textbf{\emph{i}})\\ \\
&=& \hat P(t,q) \ + \  \max\{s \log \alpha_{ \min}, s \log \alpha_{ \max}\}
\end{eqnarray*}
The proof of the left hand inequality is similar and omitted.
\\ \\
(2)  Let $q \geq 0$.  The continuity of $\hat P(s,q)$ in $s$ follows immediately from (1).
\\ \\
(3)  Let $q \geq 0$, $t \in \mathbb{R}$ and $\varepsilon > 0$. By (1) we have
\[
\hat P(s+\varepsilon, q) \leq \hat P(s, q) + \max\{\varepsilon \log \alpha_{ \min}, \varepsilon \log \alpha_{ \max}\} < \hat P(s, q) 
\]
and so $ \hat P(s,q)$ is strictly decreasing in $s$.
\\ \\
(4) It follows from (1) that $\hat P(s,q) \to \infty$ as $s \to -\infty$ and $\hat P(s,q) \to -\infty$ as $s \to \infty$.  This combined with parts (2) and (3), imply that there is a unique value of $s \in \mathbb{R}$ for which $\hat P(s,q)=0$.
\\ \\
(5) This follows easily since $\hat \Psi_k^{\gamma_k'(q), q} = 0$ and the fact that $\tfrac{1}{k} \hat \Psi_k^{s,q} \to \hat P(s,q)$ as $k \to \infty$. \hfill \qed

\subsection{Proof of Theorem \ref{generaldiffest}} \label{generaldiffestproof}

(1) Suppose $\gamma_1'(1) \leq \tau_1'(1)+\tau_2'(1)$.  This implies that there exists $\varepsilon>0$ such that for all $q \in [1,1+\varepsilon)$ we have $\gamma_1(q) \leq \tau_1(q)+\tau_2(q)$.  This means that in this interval the $\gamma_k$ are always upper estimates for $\gamma$  and are bounded above by $(\tau_1+\tau_2)$.   Thus, for all $q \in [1,1+\varepsilon)$ and all $k \in \mathbb{N}$,
\[
\tau_1(q) +\tau_2(q) \geq   \gamma_k(q) \geq \gamma(q).
\]
Since all these functions evaluate to 0 at $q=1$ it follows by taking right derivative at 1 that for all $k \in \mathbb{N}$
\[
\gamma_k'(1) = (\gamma_k)_+'(1) \geq \gamma_+'(1),
\]
which gives the result since this estimate holds for all $k \in \mathbb{N}$.
\\ \\
(2) This is proved in a similar manner to part (1) except that we use a small interval to the left of 1 and estimate the left derivative of $\gamma$.  The details are omitted. \hfill \qed

\vspace{10mm}

\addcontentsline{toc}{section}{Acknowledgements}

\begin{centering}

\textbf{Acknowledgements}

The author was supported by the EPSRC grant EP/J013560/1.  This work was started whilst the author was an EPSRC funded PhD student at the University of St Andrews and he expresses his gratitude for the support he found there. Finally he thanks Mark Pollicott for useful discussions regarding estimating pressure and De-Jun Feng for pointing out some interesting references.

\end{centering}

\end{document}